\documentclass{article}

\usepackage[utf8]{inputenc}
\usepackage{stackengine}
\usepackage{amsmath}
\usepackage{amssymb}
\usepackage{amsthm}
\usepackage[hidelinks]{hyperref}
\usepackage[shortlabels]{enumitem}
\usepackage{geometry}
\usepackage{float}
\usepackage{stmaryrd}
\usepackage{tikz-cd}

\theoremstyle{plain}
\newtheorem{thm}{Theorem}[section]
\newtheorem{lem}[thm]{Lemma}
\newtheorem{prop}[thm]{Proposition}
\newtheorem{cor}[thm]{Corollary}
\newtheorem{conj}[thm]{Conjecture}
\theoremstyle{definition}
\newtheorem{defn}[thm]{Definition}
\newtheorem{remark}[thm]{Remark}
\newtheorem{que}[thm]{Question}

\newcommand{\abs}[1]{\vert{}#1\rvert{}}
\newcommand{\Aut}{\mathrm{Aut}}
\newcommand{\Bl}{\mathrm{Bl}}
\newcommand\Brn{\mathcal{B}r_n}
\newcommand\Br{\mathrm{Br}}

\newcommand{\cI}{\mathcal{I}}
\newcommand\cL{\mathcal{L}}
\newcommand\cR{\mathcal{R}}
\newcommand{\C}{\mathbb C}
\newcommand{\defeq}{\mathrel{\vcenter{\baselineskip0.5ex \lineskiplimit0pt \hbox{.}\hbox{.}}} =}
\newcommand{\defword}[1]{\textit{#1}}
\newcommand{\D}{\mathbb{D}}
\newcommand{\eis}{E} % To distinguish notationally Eisenstein series from E_8/E_7 lattices
\newcommand{\ev}{\mathrm{ev}}
\newcommand{\ex}{\mathrm{ex}}
\newcommand{\E}[1]{\mathrm{E}_{#1}} % Cheeky way to get commands \E7 and \E8
\newcommand{\floor}[1]{\lfloor{#1}\rfloor}
\newcommand{\F}{\mathbb F}
\newcommand{\hatj}{\hat\jmath}

\newcommand{\Hm}{\mathrm{H}}

\newcommand{\hto}{\hookrightarrow}
\newcommand{\id}{\mathrm{id}}
\newcommand{\II}{\mathrm{II}}
\newcommand{\I}{^{-1}}
\newcommand{\Jac}{\mathrm{Jac}}
\newcommand{\J}{\mathcal{J}}
\newcommand{\K}{\mathrm{K}}
\newcommand{\Mb}{\overline{\M}}
\newcommand{\Mod}{\mathrm{Mod}}
\newcommand{\MRES}{\mathrm{MRES}}
\newcommand{\MW}{\operatorname{MW}}
\newcommand{\M}{\mathcal{M}}
\newcommand{\NL}{\mathrm{NL}}
\newcommand\no{\mathrm{no}}
\newcommand{\NS}{\operatorname{NS}}
\newcommand{\ol}[1]{\overline{#1}}
\newcommand{\ord}{\mathrm{ord}}

\renewcommand{\O}{\mathcal O}
\newcommand{\pb}{\ar[rd, phantom, "\lrcorner", pos=0]}
\newcommand{\PGL}{\mathrm{PGL}}
\renewcommand\phi{\varphi}
\newcommand{\Pic}{\mathrm{Pic}}
\newcommand{\Proj}{\mathrm{Proj}}
\newcommand{\pr}{\mathrm{pr}}
\renewcommand{\P}{\mathbb P}
\newcommand\qefed{=\mathrel{\vcenter{\baselineskip0.5ex \lineskiplimit0pt \hbox{.}\hbox{.}}}}
\newcommand{\QMod}{\mathrm{QMod}}
\newcommand{\Q}{\mathbb Q}
\newcommand{\reg}{\mathrm{reg}}
\newcommand{\rstr}[2]{{\ensuremath{\left.#1\right|_{#2}}}}
\newcommand{\R}{\mathbb{R}}
\newcommand{\Sch}{\mathbf{Sch}}
\newcommand{\set}[2][]{#1\{{#2}#1\}}
\newcommand{\SL}{\mathrm{SL}}
\newcommand\sm{\mathrm{sm}}
\newcommand{\Spec}{\operatorname{Spec}}
\newcommand{\Stab}{\mathrm{Stab}}
\newcommand{\Sym}{\mathrm{Sym}}
\newcommand{\tildej}{\tilde\jmath}
\newcommand{\toi}{\xrightarrow{\sim}}
\newcommand\toto{\rightrightarrows}

\newcommand{\trait}{\Delta\!\!\!\!\Delta}

\newcommand{\Wei}{\mathrm{Wei}}

\newcommand{\wt}{\widetilde}

\newcommand{\ZFF}{\mathrm{Z\hat F}}
\newcommand{\ZF}{\operatorname{ZF}}
\newcommand{\Z}{\mathbb Z}

\title{Severi curves of rational elliptic surfaces}
\author{
Fran\c{c}ois Greer \thanks{Department of Mathematics, Michigan State University. East Lansing, MI 48824. greerfra@msu.edu} \and
Joseph Helfer \thanks{Simons Center for Geometry and Physics. Stony Brook, NY 11790. joseph.helfer@stonybrook.edu} \and
John Sheridan \thanks{Department of Mathematics, Princeton University. Princeton, NJ 08544. john.sheridan@princeton.edu}
}
\date{}

\begin{document}

\maketitle

\abstract{We study Severi curves parametrizing rational bisections of elliptic fibrations associated to general pencils of plane cubics. Our main results show that these Severi curves are connected and reduced, and we give an upper bound on their geometric genus using quasi-modular forms. We conjecture that these Severi curves are eventually reducible, and we formulate a precise conjecture for their degrees in $\P^2$, featuring a divisor sum formula for collision multiplicities of branch points.}
\tableofcontents

\pagebreak

\section{Introduction}\label{sec:intro}
Let $S$ be a smooth projective algebraic surface, and $L$ a line bundle on $S$. This paper concerns certain examples of (genus 0) Severi varieties inside $|L| = \P H^0(S,L)$, defined as follows.
\begin{defn}
    The \defword{genus 0 Severi variety} $V(L) = V(S,L)$ of $L$ is the Zariski closure in $|L|$ of the locus of irreducible rational curves $C\in |L|$.
\end{defn}
The geometry of these varieties has a long history going back to Fulton, Harris, Severi, and many others. A basic question is whether $V(S,L)$ is irreducible, assuming that $\dim V(S,L)>0$. Harris answered this in the affirmative when $S=\P^2$ \cite{harris-severi}. Testa followed with an affirmative answer when $S$ is a del Pezzo surface \cite{testa}, assuming that $S$ is general in the case of degree 1 del Pezzos. We expect that Testa's theorem is sharp in the sense that it no longer holds for rational elliptic surfaces, which may be viewed as del Pezzo surfaces of degree 0.
\begin{conj}\label{conj:reducibility}
    Let $R$ be a rational elliptic surface. Then there exists a line bundle $L$ on $R$ such that the Severi variety $V(R,L)$ is reducible of dimension 1.
\end{conj}
We undertake here a systematic study of Severi curves on a general rational elliptic surface $R$, and along the way provide a heuristic argument for Conjecture \ref{conj:reducibility}. Recall that $R$ is the blow up of $\P^2$ at the 9 base points of a pencil of cubic curves. The anticanonical linear system on $R$ gives an elliptic fibration
\[\pi:R\simeq \Bl_9 \P^2 \to \P^1\]
with 12 irreducible nodal fibers. If $L$ is a globally generated line bundle on $R$, not pulled back from $\P^1$ via $\pi$, then the complete linear series $|L|$ parametrizes a family of curves with genus
\[g(L) = 1+\frac{1}{2}(L^2+L\cdot K_R).\]
Setting $l = -K_R\cdot L>0$, the Riemann-Roch Formula on $R$ combined with the Kodaira Vanishing Theorem gives:
\begin{equation}\label{eq:rr}
    r(L)=h^0(L)-1 = \frac{1}{2}(L^2+l) = g(L)+l-1.
\end{equation}
The first interesting case of a Severi variety for $R$ is when $\dim V(R,L)=1$; we call this case a {\it Severi curve}. From Equation \eqref{eq:rr}, Severi curves only appear when $l=2$, so we restrict ourselves to that situation. Since elements in the linear system $\abs{L}$ are \textit{bisections} of the anticanonical fibration $\pi$, we henceforth use the letter $B$ to denote them.
\begin{que}
What is the genus of the curve $V(R,L)$ for $L$ a bisection line bundle? If it is reducible, how many components does it have?
\end{que}
The simplest Severi curves occur when $g(L)=0$, in which case we have $|L|=V(R,L)\simeq \P^1$. For $g(L) = \frac{1}{2}L^2>0$, the Severi curve will have singularities corresponding to degenerate bisections. By construction, $V(R,L)$ embeds into $|L|\simeq \P^{g(L)+1}$, and its degree is known by Gromov-Witten theory \cite{oberdieck-pixton-elliptic}. This would give a crude upper bound on the genus by Castelnuovo's argument, but we will take a sharper approach using plane curves.
{
\renewcommand\thethm{\ref{thm:severi-connected-with-proof}}
\begin{thm}\label{thm:severi-connected}
For $(R,L)$ as above, the Severi curve $V(R,L)$ is connected.
\end{thm}
}
{
\renewcommand\thethm{\ref{thm:reducedseveri-with-proof}}
\begin{thm}\label{thm:reducedseveri}
The general element of any component of $V(R,L)$ has exactly $g = \frac{1}{2}L^2$ nodal singularities and no other singularities, and $V(R,L)$ is reduced of dimension 1.
\end{thm}
}
{
\renewcommand\thethm{\ref{thm:genusbound-with-proof}}
\begin{thm}\label{thm:genusbound}
The geometric genus of the Severi curve $V(R,L)$ is $O(g^{12+\epsilon})$, where $g = \frac{1}{2}L^2$.
\end{thm}
}
Theorem~\ref{thm:severi-connected} will be deduced from Testa's result on the irreducibility of Severi varieties for del Pezzo surfaces of degree one, combined with the Fulton-Hansen connectedness theorem.
Theorem \ref{thm:reducedseveri} uses the Nadel Vanishing Theorem and the technology of multiplier ideal sheaves. Theorem \ref{thm:genusbound} is the most involved, and we sketch the argument now.

A general element $B\in V(R,L)$ corresponds to an irreducible rational curve of degree two over $\P^1$, so we have a rational map $V(R,L) \dasharrow \Sym^2(\P^1)$, sending a bisection to its branch divisor in $\P^1$. After normalization, this map extends to a morphism
\[f_L:\widetilde{V}(R,L) \to \Sym^2(\P^1)\simeq \P^2.\]
which we show is birational onto its image. Next, the degree of this image is bounded using completed Noether-Lefschetz numbers associated to a certain family of elliptic K3 surfaces that are double covers of $R$. This gives the desired bound on the geometric genus of $V(R,L)$, since it is realized birationally as a plane curve. Lastly, we will give a conjectural formula for the precise degree, in terms of local intersection multiplicities of $f_L$ along the discriminant conic $\Delta \subset \Sym^2(\P^1) \simeq \P^2$. Geometrically, such intersections occur when the two branch points of a bisection collide.

One might compare our results with recent progress \cite{bruno-lc,chen-greer-yang} on Severi curves of a general projective K3 surface; these are loci of geometric genus {\it one} curves in the polarizing linear system. The Severi curves $V(R,L)$ considered here are isomorphic (via pullback) to subvarieties of Severi curves for special K3 surfaces of Picard rank 10. As such, our results are independent of but clearly related to the results mentioned for a general K3.

\subsubsection*{Heuristic for Conjecture \ref{conj:reducibility}.} The moduli space of stable maps $\overline{\M}_0(\P^1,2)$ is a $\mu_2$-gerbe over its coarse space $\Sym^2(\P^1)\simeq \P^2$. There is a universal family of K3 surfaces of Picard rank 10 over $\overline{\M}_0(\P^1,2)$ defined by taking the base change of $\pi:R\to \P^1$ via each stable map. This family has a corresponding period map which descends to coarse spaces as:
\[j: \P^2 \smallsetminus \Delta \to \mathbb D/\Gamma.\]

The indeterminacy locus $\Delta$ is the discriminant conic in $\Sym^2(\P^1)$. There are 12 lines in $\P^2$, each tangent to $\Delta$, over which the K3 surfaces obtained this way are nodal, but these correspond to removable singularities of $j$ away from $\Delta$. Passing to compactifications of $\D/\Gamma$, one can extend $j$ further. In this paper, we will use the Type II partial compactification:
$$j^{\II}: \P^2\smallsetminus \{12\text{ points}\} \to  (\D/\Gamma)^\II.$$

The pullback of the completed Noether-Lefschetz divisor series $\Phi^\II(q)$ via this map gives a $q$-series related to the degrees of the plane curves
$$f_L\left( \widetilde{V}(R,L) \right).$$

However, there are singularities in the image of $f_L$. These correspond set theoretically to the pre-images under $j$ of {\it codimension-2} Noether-Lefschetz cycles in $\D/\Gamma$. By W. Zhang's thesis \cite{weizhang}, the classes in $\mathrm{CH}^2(\D/\Gamma)$ of these cycles can be arranged as the $q$-expansion of a Siegel modular form. Using standard growth rates for the Fourier coefficients of Siegel theta functions, we expect the number of singularities in $f_L\left( \widetilde{V}(R,L) \right)$ to asymptotically exceed its arithmetic genus, whence reducibility of the Severi curves for $g(L) \gg 0$. To make such an argument rigorous, however, one would need to resolve the period map further to a full toroidal compactification:
\[
j^\Sigma\colon \Bl(\P^2) \to (\D/\Gamma)^\Sigma,
\]
and extend the Siegel modularity theorem of \cite{weizhang} to the closures of the special cycles of codimension 2. This has been carried out for special cycles of codimension 1 in \cite{greer-engel-tayou}, where one obtains the $q$-expansion of a mixed mock modular form. Siegel mock modular forms of genus 2 do not yet have a formal definition in the literature.

We end with a concrete example of our construction: when $g(L)=0$,
\[ V(R,L) = |L| \simeq \P^1. \]
Since the pencil has 8 singular fibers, the image of $f_L$ must be a rational quartic curve in $\P^2$. The three nodes in the image curve correspond to codimension 2 special cycles.\\\\
{\bf Acknowledgments.} We would like to extend our thanks to both Dori Bejleri and Aaron Landesman for helpful discussions while preparing this article, and to the anomyous referee for their close reading and constructive comments. The first author was partially supported by NSF grant DMS-2302548.

\section{Preliminaries}
Here we fix notation and recall some basic notions.
We work over $\C$ throughout, and we use the Fulton convention for projective space as 1-dimensional subspaces rather than quotients.

\subsection{Elliptic surfaces}\label{subsec:elliptic-surfaces}
For background on elliptic surfaces, we refer to \cite{miranda-basic-elliptic}.
Unless otherwise stated, all elliptic surfaces are assumed to be smooth and relatively minimal.
Note that an elliptic surface $X$ by definition comes equipped with a map $\pi\colon X\to C$ to a curve $C$.
Usually we will also assume that a zero section $Z\colon C\to X$ exists and has been fixed.

The basic invariant of $X$ is the degree $k$ of the pushforward $\pi_*(\mathcal N_{Z/X}^\vee)$, sometimes called the fundamental (or Hodge) line bundle.
The number of singular fibers of $X$ is then $12k$, counted with multiplicity, and $X$ is a product of $C$ with an elliptic curve if and only if $k=0$. The canonical divisor of $X$ is a sum of $2g-2+k$ fibers, where $g$ is the genus of $C$.
It follows from this and the adjunction formula that any \emph{section} of $X \to C$ has self-intersection $-k$; in particular, when $k\ge1$, all sections are rigid.

We will only encounter the case $C\simeq\P^1$, so that the fundamental line bundle is $\O_{\P^1}(k)$.
Among these, we will be concerned with those $X$ with $k=1$, which are precisely the \emph{rational elliptic surfaces}, and $k=2$, which are the \emph{elliptic K3 surfaces}.

The \emph{Mordell-Weil group} $\MW(X)$ is the group of sections $C\to X$ of $\pi$, with addition induced via the bijection of $\MW(X)\toi\MW(X_\eta)$ (given by restriction) to the Mordell-Weil group of the generic fiber $X_\eta$, which is an elliptic curve over the function field $\C(\P^1)$; on each smooth fiber, addition is simply given by the elliptic curve group law.

We denote by $\NS(X)$ the Néron-Severi group $\Pic(X)/\Pic^0(X)$ of $X$.
When $k\ge1$, we also have that $\pi^*\colon\Pic^0(C)\to\Pic^0(X)$ is an isomorphism.
In particular, in our case $C\simeq\P^1$ of interest, $\NS(X)=\Pic(X)$.
We assume for the rest of this section that $X$ has type $k\ge1$, i.e. that $X$ is not a product.

There are two important subgroups $\ZF(X)\subset\ZFF(X)\subset\NS(X)$.
$\ZF(X)$ is the subgroup generated by the zero section class $Z$ and the class $F$ of a smooth fiber, and $\ZFF(X)$ is generated in addition by all components of reducible fibers; most of the time, we will be in the situation where all fibers are irreducible and hence $\ZF(X)=\ZFF(X)$.

There is a homomorphism $\NS(X)\to\MW(X)$ taking a divisor to the closure of the sum in the generic fiber $X_\eta$ of its restriction to $X_\eta$. It is surjective and has kernel $\ZFF(X)$.
This is the \emph{Shioda-Tate exact sequence}:
\[
  0\to \ZFF(X)\to\NS(X)\to\MW(X)\to0.
\]
Note that this surjection has a natural set-theoretic section, taking each section in $\MW(X)$ to the corresponding divisor in $\NS(X)$.
The subgroup $\ZF(X)\subset\NS(X)$, equipped with the intersection form, is an indefinite unimodular lattice of rank 2.
Hence, we obtain an orthogonal decomposition $\NS(X)=\ZF(X)\oplus\ZF^\perp(X)$.

When $\ZF(X)=\ZFF(X)$ (i.e., when all fibers are irreducible), we obtain an isomorphism $\ZF^\perp(X)\toi\MW(X)$.
The inverse of this isomorphism takes a section $s\in\MW(X)$ to $\Pi([s])$, where $\Pi\colon\NS(X)\to\ZF^\perp(X)$ is the orthogonal projection.
By the Hodge index theorem, $\ZF^\perp(X)$ is always a negative definite lattice.
When $k=1$, $\ZF^\perp(X)$ is isomorphic to $\E8(-1)$.
Recall that the $\E8$ lattice is the unique rank 8 positive definite even unimodular lattice.
In general, for any lattice $E$, we write $E(n)$ for the lattice obtained by multiplying the values of the quadratic form by $n$.

\subsection{Quasi-modular forms}\label{subsec:qmod-forms}
For background on modular forms, we refer to \cite{diamond-shurman}. We will also make use of (quasi)-modular forms, but only via their Fourier expansions, so it is sufficient to view them as a special class of $q$-series ($q=e^{2\pi i \tau}$) with rational coefficients. The formal definition of quasi-modular forms as functions on the upper half plane can be found in \cite{kaneko-zagier-jacobi-and-quasimod}.

A modular form has a weight $k\in \Z_{\geq 0}$ and a level group $\Gamma\subset \SL_2(\Z)$. For each $k$ and $\Gamma$, the subset of modular forms
\[\Mod(\Gamma,k)\subset \Q\llbracket q \rrbracket\]
is a finite-dimensional $\Q$-subspace, with $\Mod(\Gamma,k)\subset\Mod(\Gamma',k)$ for $\Gamma'\subset\Gamma$. The direct sum
\[\Mod(\Gamma,*) \defeq \bigoplus_{k\geq 0} \Mod(\Gamma,k)\vspace{-3pt}\]
is a finitely generated graded $\Q$-algebra. For example, when $\Gamma=\SL_2(\Z)$, we have
\[\Mod(\SL_2(\Z),*)= \Q[\eis_4,\eis_6]. \]
The modular form $\eis_4$ plays a central role because it is the theta function of the $\E8$ lattice (see below).
Below we list the normalized Eisenstein series with weights indicated by the subscripts, where $\sigma_i$ is the $i$-th divisor sum function:
\begin{align*}
    \eis_2(q) &= 1-24 \sum_{n\geq 1} \sigma_1(n)q^n;\\
    \eis_4(q) &= 1+240 \sum_{n\geq 1} \sigma_3(n)q^n;\\
    \eis_6(q) &= 1-504 \sum_{n\geq 1} \sigma_5(n)q^n.
\end{align*}
The weight 2 Eisenstein series is not quite a modular form; it belongs to a slightly larger class of $q$-series called quasi-modular forms. Following \cite{kaneko-zagier-jacobi-and-quasimod}, we have:
\begin{defn}
  For any level group $\Gamma\subset \SL_2(\Z)$, the algebra $\QMod(\Gamma,*)\subset\Q\llbracket q\rrbracket$ of quasi-modular forms is the subalgebra of $\Q\llbracket q\rrbracket$ generated by $\Mod(\Gamma,*)$ and $\eis_2$.
\end{defn}
\noindent
In fact, $\QMod(\Gamma,*)$ is freely generated over $\Mod(\Gamma,*)$ by $\eis_2$, and it has an obvious weight grading:
  \[\QMod(\Gamma,*) \simeq \Mod(\Gamma,*)\otimes \Q[\eis_2].\]
\noindent
Quasi-modular forms are closed under a natural differential operator
\[D= q\frac{d}{dq}: \QMod(\Gamma,k) \to \QMod(\Gamma,k+2).\]
For example, we will use the following identity due to Ramanujan:
\begin{equation} \label{eq:ramanujan} D\eis_4 = \frac{\eis_2\eis_4 - \eis_6}{3}.\end{equation}
It is easily verified that if $F(q)\in \Mod(\SL_2(\Z),k)$, then $F(q^m)\in \Mod(\Gamma_0(m),k)$. For example,
\[\eis_{4,\infty}(q)\defeq \frac{1}{240}\left(\eis_4(q) - \eis_4(q^2)\right)\in \Mod(\Gamma_0(2),4)\]
appears in Section \ref{sec:multiplicity-conjecture}. Finally, we introduce the Jacobi theta function of weight 1/2:
\[\theta(q) \defeq 1+2 \sum_{n\geq 1} q^{n^2},\]
which appears as a correction term from the nodal K3 surfaces in Section \ref{sec:bounding}.

The divisor sum function $\sigma_k(n)$ satisfies the bound $n^k\le\sigma_k(n)\le Cn^k$ for $k>1$, and $\sigma_1(n)\le Cn^{1+\varepsilon}$ for all $\varepsilon>0$.
In general, the coefficients $a_n$ of a weight $k$ quasi-modular form satisfy the ``trivial bound'' $a_n=O(n^{k-1+\varepsilon})$.
% A reference for this: http://www.math.columbia.edu/~phlee/CourseNotes/ModularForms.pdf}

\subsection{The $\E8$ lattice}\label{subsubsec:e8-lattice}
The positive definite lattice $\E8$ can be defined explicitly as the set of integer or half-integer vectors in $\R^8$ whose coefficients sum to an even integer:
\[
  \E8=\set[\Big]{\,\vec{v}\in \Z^8\cup(\tfrac12\Z^8\setminus\Z^8)\Bigm\vert \sum_{i=1}^8v_i\in 2\Z\,}.
\]
We have:
\begin{prop}\label{prop:e8-and-e4}
  The theta function of the \( \E8 \)-lattice is the normalized Eisenstein series \( \eis_4 \) -- i.e., the coefficient of \( q^n \) in \( \eis_4 \) is the number of vectors of norm \( 2 n \) in \( \E8 \).
\end{prop}
\begin{proof}
  See, e.g., \cite[VII.6.6 (i)]{serre-course}.
\end{proof}

We denote by $\Aut(\E8)$ the group of linear isometries.

\begin{lem}\label{lem:trans-on-roots}
$\Aut(\E8)$ acts transitively on the vectors of norm 2 and on those of norm 4.
\end{lem}
\begin{proof}
The vectors of norm 2 are the roots of the $\E8$ root system, and in general, the Weyl group of any simply laced root system acts transitively on the roots \cite[Lemma~10.4.C]{humphreys-lie-rep}.

As for the norm 4 vectors, one verifies that these are precisely given by plus or minus a permutation of the following vectors (where superscripts denote repeated entries)
\[
(2,0^7)
\ \ (1^4,0^4)
\ \ (1^3,-1,0^4)
\ \ (1^2,-1^2,0^4)
\ \ (\tfrac32,\tfrac12^6,-\tfrac12)
\ \ (\tfrac32,\tfrac12^4,-\tfrac12^3)
\ \ (\tfrac32,\tfrac12^2,-\tfrac12^5)
\ \ (\tfrac32,-\tfrac12^7)
\]
whereas the roots are (again up permutation and sign)
\[
(1^2,0^6)\quad
(1,-1,0^6)\quad
(\tfrac12^8)\quad
(\tfrac12^2,-\tfrac12^6)\quad
(\tfrac12^4,-\tfrac12^4).
\]
From this, one checks directly that each norm 4 vector is a sum of two roots.

Since we already know that $\Aut(\E8)$ acts transitively on the roots, it therefore suffices to check that it acts transitively on norm 4 vectors of the form $r+(\frac12^8)$, with $r$ a root.
These are all permutations of $(\frac32,\frac12^6,-\frac12)$ and $(1^4,0^4)$.
But $\Aut(\E8)$ includes all permutations (the transpositions being reflections along the root $(1,-1,0^6)$ and its permutations), and $(\frac32,\frac12^6,-\frac12)$ is the reflection of $(1^4,0^4)$ along the root $(-\frac12,\frac12^3,-\frac12^3,\frac12)$.
\end{proof}

\begin{lem}\label{lem:parities-classf}
Each non-zero equivalence class in $\E8/2\E8$ contains a vector of norm 2 and consists entirely of vectors $v$ with norm $\equiv 2\mod 4$, or contains a vector of norm 4 and consists entirely of vectors with norm $\equiv 0\mod 4$.
\end{lem}
\begin{proof}
First we note that if $u,v\in \E8$ are equal in $\E8/2\E8$, then $u^2\equiv v^2\mod 4$ since, setting $u-v=2w$, we have $u^2=(v+2w)^2=v^2+4v\cdot w+4w^2$.

Thus, it only remains to show that each 2-prim vector $u$ is equal mod $2\E8$ to some $v$ with $v^2\in\set{2,4}$. This can be seen by induction on the sum of the absolute values of the entries of $u$; for the base case, if this sum is $\le 4$ (and $u$ is 2-prim), then $u^2\le 4$.
\end{proof}

For the rest of this section, let $R\to\P^1$ be a rational elliptic surface with no reducible fibers.
We compute numerical invariants of sections and bisections of $R$ in terms of their projections to $\ZF^\perp(R)\cong\E8(-1)$ (see \S\ref{subsec:elliptic-surfaces}).

\begin{defn}
  The \emph{height} of a section $\P^1\to R$ (or more generally, of any curve in $R$) is its intersection number with the zero-section $Z$.
\end{defn}

\begin{prop}\label{prop:sec-proj-formula}
  Let $s,t\in\NS(R)$ be sections of heights $m$ and $n$ respectively. Then
  \[
    \Pi(s)\cdot\Pi(t) =
    s\cdot t - 1 - (m+n).
  \]
  In particular, $\Pi(s)^2=-2-2m$.
\end{prop}
\begin{proof}
  We have $Z^2=-1$, $F^2=0$, and $Z\cdot F=1$. Hence, $Z$ and $\widetilde{F}=Z+F$ form an orthogonal basis of $\ZF(R)$. Hence, the orthogonal projection can be computed as $\Pi(s)=s-\frac{s\cdot Z}{Z^2}Z-\frac{s\cdot \widetilde{F}}{\widetilde{F}^2}\widetilde{F}$, from which the equation follows.
\end{proof}

\begin{prop}\label{prop:bisec-proj-formula}
  Given a bisection $B\in\NS(R)$ of $R$ of height $n$ and genus $g$, we have
  \[
    \Pi(B)^2=2g-(4n+4).
  \]
\end{prop}
\begin{proof}
  It follows from the adjunction formula that $B^2=2g$. Now proceed as in Proposition~\ref{prop:sec-proj-formula}.
\end{proof}

\begin{lem}\label{lem:g-bisections-inject}
  Let $\mathcal{B}_g\subset\NS(R)$ be the subset consisting of divisor classes of arithmetic genus $g$ bisections.
  Then the restriction $\rstr{\Pi}{\mathcal{B}_g}:\mathcal{B}_g\to\ZF^\perp(R)$ is injective.
\end{lem}
\begin{proof}
  If $\Pi(B_1)=\Pi(B_2)$, then $B_1$ and $B_2$ differ by some linear combination $aZ+bF$. We must have $a=0$ if $B_1$ and $B_2$ are both bisections, and we must then have $b=0$ if $B_1$ and $B_2$ are to have the same arithmetic genus.
\end{proof}

\begin{lem}\label{lem:g-g-bisections-biject}
  Let $\mathcal{B}_{g,n}\subset\mathcal{B}_g$ be the set of divisor classes of arithmetic genus $g$, height $n$ bisections. Then the restriction $\rstr{\Pi}{\mathcal{B}_{g,g}}:\mathcal{B}_{g,g}\to\ZF^\perp(R)$ is a bijection onto the set of $u\in\ZF^\perp(R)$ of norm $-2(g+2)$.

  In particular, by Proposition~\ref{prop:e8-and-e4}, $\#\mathcal{B}_{g,g}$ is the coefficient of $q^{g+2}$ in $\eis_4(q)$.
\end{lem}
\begin{proof}
  We have injectivity by Lemma~\ref{lem:g-bisections-inject}. For surjectivity, given $u\in\ZF^\perp(R)$, let $s$ be a section with $\Pi(s)=u$, so $s$ has height $g+1$.
  Then $s+Z$ is a bisection of height $(g+1)-1=g$ and genus $s\cdot Z-1=g$ as desired.
\end{proof}

We will also need the following variant of the above lemma; both variants will be generalized in  Corollary~\ref{cor:g-n-bisections-biject}.
See \S\ref{subsec:moduli_of_rl} for the definition of Weierstrass bisection line bundles.

\begin{lem}\label{lem:g-g-1-bisections-biject}
 Let $\mathcal{B}_{g,n}^\Wei\subset\mathcal{B}_{g,n}$ be the subset consisting of bisections with $L=\O(B)$ Weierstrass.
 Then the restriction $\rstr{\Pi}{\mathcal{B}^{\Wei}_{g,g-1}}:\mathcal{B}^{\Wei}_{g,g-1}\to\ZF^\perp(R)$ is a bijection onto the set of 2-divisible $u\in \ZF^\perp(R)$ of norm $-2g$.

  In particular, by Proposition~\ref{prop:e8-and-e4}, $\#\mathcal{B}^{\Wei}_{g,g-1}$ is the coefficient of $q^{g}$ in $\eis_4(q^4)$.
\end{lem}

\begin{proof}
  We again have injectivity by Lemma~\ref{lem:g-bisections-inject}. For surjectivity, given $u\in\ZF^\perp(R)$ with $u=2v$, let $s$ be a section with $\Pi(s)=v$, so $s$ has height $\frac g4-1$.
  Then $2s+(\frac g2+1)F$ is a bisection of height $g-1$ and genus $g$ as desired.
\end{proof}

\section{Severi varieties}
Let $L$ be a line bundle on a surface $S$. Recall that for us, the \emph{Severi variety} $V(S,L)\subset\abs{L}$ is the closure in $\abs{L}\cong\P^r$ of the locus of irreducible rational curves in $\abs{L}$.

The \emph{Severi problem} asks whether $V(S,L)$ is irreducible.
Harris \cite{harris-severi} proved this in the case $S=\P^2$, and more recent work of Testa \cite{testa} proved it for $S$ any del Pezzo surface of degree $\ge2$, and for a general del Pezzo surface of degree 1 with $L\ne\omega_S\I$.

The expected dimension of a Severi variety $V(S,L)$ for line bundles $L$ with vanishing higher cohomology is given by the genus formula and Riemann Roch:
\begin{equation}\label{eq:expdimformula}
\begin{split}
r(L)-g(L) &= \frac{L^2 - K_S\cdot L}{2} + \chi(\O_S) - 1 - \frac{L^2 + K_S\cdot L}{2}-1\\
&= -K_S\cdot L - 2 + \chi(\O_S).
\end{split}
\end{equation}
The actual dimension agrees with this expected dimension in the cases of del Pezzo surfaces \cite{testa} and of rational elliptic surfaces (see Proposition \ref{expdimseveri}).

\subsection{Connectedness of $V(R,L)$}
We now prove a weak analogue of Testa's result in the case of rational elliptic surfaces $R$.
This is logically independent from the rest of the paper. The strategy is to use the irreducibility of Severi varieties for general degree 1 del Pezzo surfaces, combined with a connectedness theorem of Fulton-Hansen \cite[Corollary~1]{fulton-hansen}: if $X,Y\subset \P^N$ are irreducible varieties such that $\dim(X)+\dim(Y)>N$, then $X\cap Y$ is connected.

\begin{thm}\label{thm:severi-connected-with-proof}
Suppose $R$ is a general rational elliptic surface.
Then $V(R,L)$ is connected for any line bundle $L\ne\omega_R\I$ on $R$.
\end{thm}

\begin{proof}
Let $Z$ be a section of $R$, which is a $(-1)$-curve by the genus formula.
Blowing down $R$ at $Z$, we obtain a general degree one del Pezzo surface $S$, as explained in \cite{ascher-bejleri-moduli-of-del-pezzo}.
Let $\epsilon\colon R\to S$ be the blowdown map and $p=\pi(Z)\in S$. The proper pushforward on 1-cycles gives an injective map
\[\epsilon_* \colon |L| \to |\epsilon_* L|. \]
Since the image of a rational curve remains rational, $\epsilon_*$ restricts to an embedding of Severi varieties:
\[ \epsilon_*\colon V(R,L) \to V(S,\epsilon_*L). \]
Let $n$ be the intersection number $L\cdot Z$.
A general element $C$ in $V(R,L)$ is irreducible, so it does not contain $Z$ as a component unless $L=\O_R(Z)$.
The image of $\epsilon_*$ consists of curves in $|\epsilon_*L|$ which pass through $p$ with multiplicity $\geq n$.
This is a linear condition on $|\epsilon_*L|$ of codimension $\le \binom{n+1}{2}$. However, the codimension of the image $\epsilon_* (V(R,L)) \subset V(S,\epsilon_*L)$ is only $n$. Hence the hypotheses of the Fulton-Hansen connectedness theorem are not fulfilled.
%(The Lefschetz hyperplane theorem, which likewise could be used to show connectedness, equally does not imply since we do not have control over the smoothness of $V(\epsilon_*L)$.)

Following \cite{testa}, let $\Mb_0^{\mathrm{bir}}(S,\epsilon_*L)$ be the closure in $\Mb_0(S,\epsilon_*L)$ of the locus of stable maps that are birational onto their image. Testa proves that this $\Mb_0^{\mathrm{bir}}(S,\epsilon_*L)$ is irreducible, and its coarse space is birational to the Severi variety $V(S, \epsilon_*L)$. Consider now the moduli stack of stable maps with $n$ marked points, $\Mb_{0,n}(S,\epsilon_*L)$, and let $\Mb_{0,n}^{\mathrm{bir}}(S,\epsilon_*L)$ be the pre-image $\pi^{-1}\left(\Mb_{0}^{\mathrm{\mathrm{bir}}}(S,\epsilon_*L)\right)$, where $\pi$ is the forgetful map
\[ \pi:\Mb_{0,n}(S,\epsilon_*L) \to \Mb_{0}(S,\epsilon_*L) \]
Observe that $\Mb_{0,n}^{\mathrm{bir}}(S,\epsilon_* L)$ is irreducible because it can be constructed iteratively by taking the (flat) universal family of rational curves whose general fiber is irreducible. We have an evaluation map $\ev: \Mb_{0,n}^{\mathrm{bir}}(S,\epsilon_*L) \to S^n$ which takes the image of the stable map at the $n$ marked points.
A general fiber $\ev^{-1}({\bf s})$ for ${\bf s}\in S^n$ is birational to a codimension $n$ linear section of $V(S,\epsilon_*L)$.
This is the intersection of the irreducible projective variety $V(S,\epsilon_*L) \subset \abs{\epsilon_* L}$ with a linear subspace, which is likewise irreducible, and we may thus deduce that it is connected by the Fulton-Hansen theorem. The dimension assumption $\dim V(S,\epsilon_*L) = n + \dim V(L) > n$ of that theorem is satisfied by \eqref{eq:expdimformula}.

Since connectedness of fibers specializes in an irreducible family, it follows that the fiber $\ev^{-1}(p,p,\dots,p)$ is connected. The image of this fiber under the cycle map $\Mb_{0,n}^{\mathrm{bir}}(S,\epsilon_*L) \to V(S,\epsilon_*L)$ is also connected, and this is precisely $\epsilon_*(V(R,L))$.
\end{proof}

\subsection{Reducedness of $V(R,L)$}

Let $R$ be a general rational elliptic surface, and let $B\subset R$ be a rational bisection parametrized by a general point of any component of the Severi curve $V(R,L)$.
We prove that $B$ has only nodal singularities, and moreover that it has a 1-dimensional space of first order equisingular deformations.
A key observation is that the local defining equation of a bisection near any singular point must be $y^2 = x^d$, with $d\geq 2$, so it suffices to rule out $A_{d-1}$ singularities for $d\geq 3$ on a general $B\in V(L)$. We assume throughout this section that $L=\O(B)$ is ample. This is true when $g(L)>0$; for $g(L)=0$ we have $V(L)=|L|\simeq \P^1$, and the results of this section hold trivially.

\begin{prop}\label{prop:jacobianideal}
    If $s_1,\dots, s_n\in B$ are isolated singularities with local analytic equations $y^2=x^{d_i}$, then the first order equisingular deformations of $B$ in its linear system are given by
    $$\Hm^0(B,L_B\otimes \Jac(B))\subset \Hm^0(B,L_B),$$ where $\Jac(B)$ is the Jacobian ideal cosupported on $\{s_1,\dots,s_n\}$ defined locally by $(y,x^{d_i-1})$.
\end{prop}
\begin{proof}
    Since the $A_{d-1}$ singularities are defined by quasi-homogeneous polynomials, the equisingular ideal coincides with the Jacobian ideal, which is generated by the partial derivatives of the equation of the singularity \cite{dedieu-sernesi}. See \cite{harris-diaz-ideals} for the deformation theory of plane curve singularities.
\end{proof}
Recall that the $\delta$-invariant of the planar singularity $y^2=x^d$ is equal to $\delta(d)=\lfloor d/2 \rfloor$. Thus, if $B$ is an irreducible rational bisection, then
\begin{equation} \label{eq:genus-sum-of-deltas}
\sum_{i=1}^n \lfloor d_i/2 \rfloor = g(L).
\end{equation}
The following lemma is a vanishing theorem concerning a larger ideal sheaf $\mathcal I$ containing $\Jac(B)$.

\begin{lem}\label{lem:vanishingtheorem}
    Let $L$ be a bisection line bundle on $R$, and let $B\in\abs{L}$ be a rational bisection. Consider the ideal sheaf $\cI\subset\O_B$ cosupported on the singularities $\{s_1,\dots,s_n\}$ of $B$ defined locally as follows. If $s_i$ is a node of $B$, then $\mathcal I_{s_i} = (y,x)$. If $s_i$ is a singularity of type $y^2 = x^{d_i}$ with $d_i\geq 3$, then $\mathcal I_{s_i} = (y,x^{\delta(d_i)+1})$. Then $\Hm^1(B,L_B\otimes \mathcal I)\cong 0$.
\end{lem}

\begin{proof}
  We will employ the theory of multiplier ideals and the Nadel Vanishing Theorem (see e.g. \cite[Theorem~9.4.8]{lazarsfeld-positivity_2}), which states that given a line bundle $L$ and an effective $\Q$-divisor $D$ such that $L\otimes\O(-D)$ is nef and big, the \emph{multiplier ideal} $\J(D)$ associated to $D$ satisfies
  $$\Hm^i(R,\O(K_R)\otimes L\otimes\J(D))\cong0$$
  for $i>0$. We claim that it suffices to find an effective $\Q$-divisor $D$ such that
  \begin{enumerate}[(i)]
  \item\label{item:recipe-ample} $L\otimes\O(F-D)$ is nef and big (where $F$ is the fiber class);
  \item\label{item:recipe-cosupport} $\J(D)$ has finite cosupport;
  \item\label{item:recipe-inclusion} $\J(D)\subset \cI$.
  \end{enumerate}
  Indeed, since $K_R=\O(-F)$, \ref{item:recipe-ample} combined with the Nadel Vanishing Theorem implies that $\Hm^i(L\otimes \J(D))\cong 0$ for $i>0$. Now consider the following exact sequence obtained by tensoring the ideal sequence for $B \subset R$ with $L\otimes \mathcal{J}(D)$:
  $$
  0 \longrightarrow \operatorname{\underline{Tor}}_1^{\mathcal{O}_R}(\mathcal{O}_B,L\otimes\mathcal{J}(D)) \longrightarrow L(-B)\otimes \mathcal{J}(D) \longrightarrow L\otimes \mathcal{J}(D) \longrightarrow L_B\otimes \mathcal{J}(D) \longrightarrow 0.
  $$
  If $\mathcal{F}$ denotes the kernel of the restriction $L\otimes \mathcal{J}(D) \to L_B\otimes \mathcal{J}(D)$ then since $H^1(L\otimes \mathcal{J}(D)) = 0$ we have an injection
  $$
  H^1(L_B\otimes \mathcal{J}(D)) \hookrightarrow H^2(\mathcal{F})
  $$
  and, since $R$ has dimension 2, we also have a surjection
  $$
  H^2(L(-B)\otimes \mathcal{J}(D)) \twoheadrightarrow H^2(\mathcal{F}).
  $$
  So if $H^2(L(-B)\otimes \mathcal{J}(D)) = 0$ then we would have that $H^1(L_B\otimes \mathcal{J}(D)) = 0$. But $H^2(L(-B)\otimes \mathcal{J}(D))$ sits in the exact sequence
  $$
  H^1(L(-B)\otimes \mathcal{O}_R/\mathcal{J}(D)) \longrightarrow H^2(L(-B)\otimes \mathcal{J}(D)) \longrightarrow H^2(L(-B))
  $$
  and therefore must be 0 since the left term vanishes by \ref{item:recipe-cosupport}, and the right term vanishes because $L(-B) \cong \mathcal{O}_R$ (and the structure sheaf of a smooth rational surface has no higher cohomology). Thus,
  $$
  H^1(L_B\otimes \mathcal{J}(D)) = 0.
  $$
  But now by \ref{item:recipe-inclusion}, we have an inclusion $L_B\otimes\J(D)\hto L_B\otimes\cI$, and by \ref{item:recipe-cosupport}, we get that $\Hm^1(B,L_B\otimes \cI/\J(D))\cong0$, since $\mathcal{I}/\mathcal{J}(D)$ must be supported in dimension 0. The associated long exact sequence gives a surjection
  $$\Hm^1(B,L_B\otimes\J(D))\twoheadrightarrow\Hm^1(B,L_B\otimes\cI).$$
  and so $\Hm^1(B, L_B\otimes\cI)\cong0$, as desired.

  It now remains to find $D$ satisfying \ref{item:recipe-ample}-\ref{item:recipe-inclusion}.
  We will take $D=aB+b\Gamma$, with $a,b\in\Q_{>0}$ to be determined, where $\Gamma\in\abs{B+mF}$ is a curve to be described presently ($m\gg 0$ an integer). In terms of these constants, condition \ref{item:recipe-ample} is the nef-and-bigness of
  $$(1-a-b)L + (1-bm)F,$$
  which is implied by $a+b<1$ and $b<1/m$ using the numerical criterion: it is a sum of two nef divisors, and $L^2 = 2g$. Note also that $0<a,b<1$ implies condition \ref{item:recipe-cosupport} because the multiplier ideal $\mathcal{J}(D)$ is trivial precisely on the klt locus of the pair $(R,D)$; see \cite[Definition 9.3.9, Remark 9.3.11]{lazarsfeld-positivity_2}. The pair $(R,D)$ can only fail to be klt along the singular locus of the support of $D$, which is finite.

  We will take $\Gamma\in\abs{B+mF}$ to be any reduced divisor, not containing $B$, which passes through all the singularities of $B$, with multiplicity 3 at each nodal singularity, and with an $A_{d}$ singularity at each $A_{d-1}$ singularity of $B$ for $d\geq 3$, with the same preferred tangent direction. The existence of such a $\Gamma$ can be arranged from the realization of $(R,L)$ as the double cover of a Hirzebruch surface, as explained in the next section, detailed in Lemma \ref{lem:gamma}.

  To verify condition \ref{item:recipe-inclusion}, we use the definition of the multiplier ideal $\J(D)$. Take a log resolution $\mu:\widetilde{R}\to R$ of $(R,D)$, and recall that
  $$\J(D)\defeq\mu_* \O_{\widetilde{R}}\left(K_{\widetilde{R}/R} - \lfloor \mu^*D \rfloor \right).$$
  Observe that the log resolution for the $A_{d-1}$ plane curve singularity depends on the parity of $d$. The incidence and multiplicities of the proper transform $\widetilde{C}$ of the singular curve $C=Z(y^2-x^d)$ and the exceptional divisors of the resolution are recorded by the graphs below.
  
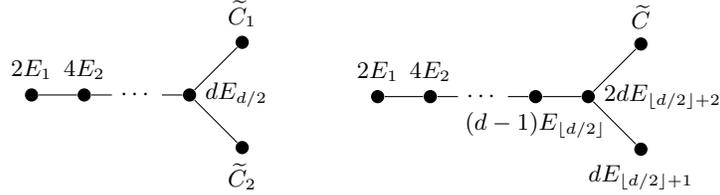
\[
  \begin{tikzpicture}[dot/.style={circle,fill=black,inner sep=0pt, minimum size=5pt}, scale=0.7]
  \node[dot,label=above:{$2E_1$}] (n1) at (0,0) {};
  \node[dot,label=above:{$4E_2$}] (n2) at (1,0) {};
  \node (n3) at (2,0) {$\cdots$};
  \node[dot,label=right:{$dE_{d/2}$}] (n4) at (3,0) {};
  \node[dot,label=above:{$\wt C_1$}] (n5) at (4,1) {};
  \node[dot,label=below:{$\wt C_2$}] (n6) at (4,-1) {};
  \draw (n1) -- (n2) -- (n3) -- (n4) -- (n5) (n4) -- (n6);
  \end{tikzpicture}
  \quad\quad\quad
  \begin{tikzpicture}[dot/.style={circle,fill=black,inner sep=0pt, minimum size=5pt}, scale=0.7]
  \node[dot,label=above:{$2E_1$}] (n1) at (0,0) {};
  \node[dot,label=above:{$4E_2$}] (n2) at (1,0) {};
  \node (n3) at (2,0) {$\cdots$};
  \node[dot,label=below:{$(d-1)E_{\floor{d/2}}$}] (n4) at (3,0) {};
  \node[dot,label=right:{$2dE_{\floor{d/2}+2}$}] (n5) at (4,0) {};
  \node[dot,label=above:{$\wt C$}] (n6) at (5,1) {};
  \node[dot,label=below:{$dE_{\floor{d/2}+1}$}] (n7) at (5,-1) {};
  \draw (n1) -- (n2) -- (n3) -- (n4) -- (n5) -- (n6) (n5) -- (n7);
  \end{tikzpicture}
\]

  Now $K_{\widetilde{R}/R}= \sum E_i$, and the ideal sheaf $(y,x^k)$ is the pushforward of $\O_{\widetilde{R}}(-kE_k)$. Since pushforward is left exact, it suffices to check that
  \begin{equation} \label{eq:ineq-at-each-sing}
  \begin{aligned}
  &K_{\widetilde{R}/R} - \lfloor \mu^*D \rfloor \leq -(\delta(d)+1) E_{\delta(d)+1}
  & d &\geq 3
  \\
  &K_{\widetilde{R}/R} - \lfloor \mu^*D \rfloor \leq -E_{1} & d &= 2
  .
  \end{aligned}
  \end{equation}
 We arrange this by comparing coefficients at each exceptional divisor $E_i$. Suppose first that $d \geq 3$. For $i\neq \delta(d)+1$, the necessary inequality for $a,b$ becomes
$$1- \lfloor a (2i)+ b(2i) \rfloor \leq  0.$$
 This statement for all $i$ would follow from $a+b>1/2$. For $i=\delta(d)+1$, the necessary inequality for $a,b$ is
  $$1- \lfloor da + (d+1)b \rfloor \leq  -(\delta(d)+1)$$
  $$ \Leftrightarrow \,\, da + (d+1)b \geq \lfloor d/2 \rfloor +2.$$
For $d=3$, this becomes $3a+4b\geq 3$. For $d=4$, this becomes $4a+5b\geq 4$. For $d=5$, this becomes $5a+6b\geq 4$, and so on.

When $d=2$, the case of the node singularity, \eqref{eq:ineq-at-each-sing} yields the weaker inequality $2a+3b\geq 2$. All the desiderata are satisfied when $b$ is sufficiently small, and
\[ 
a = 1-\frac{5}{4}b. \qedhere
\]
\end{proof}

\begin{thm}\label{thm:equisingular}
Let $B\in |L|$ be an irreducible rational bisection curve. Then $B$ has at most one singularity that is worse than nodal. If $B$ has such a singularity, then the space of first order equisingular deformations of $B$ is trivial, and otherwise it is 1-dimensional.
\end{thm}

\begin{proof}
  Let \( \cI \) be as in Lemma~\ref{lem:vanishingtheorem}.
  Consider the following short exact sequence of sheaves on $B$:
  $$0 \to L_B\otimes \cI \to L_B \to \O_B/\cI \to 0.$$
  The long exact sequence on cohomology then reads
  $$0 \to \Hm^0(L_B\otimes \cI) \to \Hm^0(L_B) \to \Hm^0(\O_B/\cI) \to \Hm^1(L_B\otimes \cI).$$
  Recall that $h^0(L_B)=g+1$, and by definition $\O_B/\cI$ has length $g+\ell$, where $\ell$ is the number of worse-than-nodal singularities by \eqref{eq:genus-sum-of-deltas}.  By Lemma \ref{lem:vanishingtheorem}, the $\Hm^1$ term vanishes, so $\ell\leq 1$ by exactness.
  When $\ell=0$, we conclude that $h^0(L_B\otimes \cI)=h^0(L_B\otimes\Jac(B))=1$.
When $\ell=1$, we conclude that $h^0(L_B\otimes \cI)=0$. Now, since the Jacobian ideal $\Jac(B)$ is a subsheaf of $\cI$, the equisingular deformation space is trivial by Proposition \ref{prop:jacobianideal}.
\end{proof}

While Theorem \ref{thm:equisingular} is interesting in its own right, it is useful for our present purposes because it implies a certain transversality statement for Noether-Lefschetz intersections.

\begin{thm}\label{thm:reducedseveri-with-proof}
    The general element of any component of $V(R,L)$ has exactly $g = \frac{1}{2}L^2$ nodal singularities and no other singularities, and $V(R,L)$ is reduced of dimension 1.
\end{thm}
\begin{proof}
    Each component of the Severi variety $V(R,L)$ has dimension $\geq 1$ by Theorem \ref{thm:severi-connected-with-proof}; see also Proposition \ref{expdimseveri}. Assume for the sake of contradiction that a general element $B$ of some component of $V(R,L)$ had a singularity of type $A_d$ where $d\geq 2$. By Theorem \ref{thm:equisingular}, $B$ has no first order equisingular deformations; contradiction! Therefore, $B$ has only $A_1$ singularities, i.e., nodes. The number of nodes is equal to the arithmetic genus, which is $\frac{1}{2}L^2$ by the genus formula. Since the first order equisingular deformations of $B$ are 1-dimensional, $V(R,L)$ has dimension equal to 1, and it is reduced.
\end{proof}

\subsection{Moduli Spaces of $(R,L)$}\label{subsec:moduli_of_rl}
The purpose of this section is to show that most geometric properties of the Severi curve $V(R,L)$ depend only on the value of $L^2 =2g(L)$, and the dichotomy of whether $L$ is \emph{Weierstrass} or \emph{ordinary}, to be explained below.
We will parametrize all pairs $(R,L)$, at least for $R$ smooth with irreducible fibers, by a countable collection of irreducible varieties (or stacks), the only discrete invariants being the aforementioned. A similar analysis is carried out for higher degree multisections in \cite{dejong-friedman}. The moduli stack of rational elliptic surfaces with no additional data is poorly behaved because the generic $R$ has infinite discrete automorphism group.

To explain the dichotomy, let $\eta \to \P^1$ be the generic point and $R_\eta \to \eta$ the pullback, which is an elliptic curve over the function field $K=\C(\P^1)$. The line bundle $L$ pulls back to $L_\eta \in \Pic^2(R_\eta)$. By Riemann-Roch for the curve $R_\eta$, $h^0(L_\eta)=2$ so we have a $K$-morphism
\[|L_\eta|:R_\eta \to \P^1_K.\]
Over the algebraic closure $\overline{K}$, this morphism has 4 ramification points, but they are rarely $K$-points. We briefly study when this is the case.

\begin{defn}
    We say that $L$ is \defword{Weierstrass} if there is a ramification point of $\abs{L_\eta}: R_\eta \to \P^1_K$ defined over $K$; otherwise $L$ is \defword{ordinary}.
\end{defn}
\begin{prop}
    Let $\pi:R \to \P^1$ be a rational elliptic surface with irreducible fibers, and $L$ a bisection line bundle. Then the following conditions are equivalent.
    \begin{enumerate}[(i)]
        \item $L$ is Weierstrass.
        \item The complete linear system $\abs{L}$ contains a divisor of the form $2s+mf$, for $s$ a section of $\pi$ and $f$ a fiber of $\pi$.
        \item The projection $\Pi(L)\in\ZF^\perp(R)\cong\E8(-1)$ is 2-divisible, i.e., $\Pi(L)=2u$ for some $u\in\ZF^\perp(R)$ for any choice of zero section as in \S\ref{subsec:elliptic-surfaces}.
    \end{enumerate}
\end{prop}
\begin{proof}
    (i)$\Leftrightarrow$(ii): If the double cover $|L_\eta|:R_\eta \to \P^1_\eta$ has a ramification point defined over $K$, then taking its Zariski closure gives a section $s$ on $R$. The associated branch point in $\P^1_K$ is also defined over $K$, so the pullback of divisors gives a rational equivalence $L_\eta \sim 2s$ on $R_\eta$. Since any divisor on $R$ supported on fibers of $\pi$ is a sum of fibers, we obtain that $L \sim 2s+mf$. Conversely, if $|L|$ contains a divisor $2s+mf$, then $L_\eta$ contains the divisor $2s_\eta$, and $s_\eta$ is a $K$-point of $R_\eta$ where $|L_\eta|$ ramifies.

    (ii)$\Leftrightarrow$(iii): If $|L|$ contains the divisor $2s+mf$, then since $\Pi(f)=0$, 
    $$\Pi(L) = 2\Pi(s)\in \ZF^\perp(R).$$
    Conversely, if $\Pi(L)=2u$, then as $\Pi$ is an isomorphism from the Mordell-Weil group to $\ZF^\perp(R)$, there is a unique section $s$ such that $\Pi(s)=u$. By Proposition \ref{prop:bisec-proj-formula}, $L^2 - \Pi(L)^2\equiv 0$ (mod 4), so there is a unique $m\in \Z$ such that $(2s+mf)^2 = L^2$. We conclude using Lemma \ref{lem:g-bisections-inject}.
\end{proof}

Let $\pi\colon R\to\P^1$ be a rational elliptic surface with irreducible fibers and $L$ either an ample line bundle (if $L^2>0$) or a nef line bundle (if $L^2=0$), and $L\cdot F=2$. These conditions rule out bisection line bundles of the form $\O(s_1+s_2)$ for $s_1\cap s_2=\emptyset$ and $\O(2s_1+F)$, which have an empty Severi variety.

\begin{prop}
    The restriction map $\Hm^0(R,L) \to \Hm^0(f,L_f)$ is surjective for any fiber $f$ of the anticanonical fibration $\pi:R \to \P^1$.
\end{prop}
\begin{proof}
    If $L$ is ample, then the claim follows directly from Kodaira vanishing. If $L^2=0$ and $L$ is nef, then $L+F$ is nef and big, so by Kawamata-Viehweg vanishing and Riemann-Roch, we have that $h^0(L)=2$. Hence, $|L|\simeq \P^1$ and the corresponding pencil of bisection curves is basepoint free:  there can be no fixed component since it would have to be a section curve, and there are no isolated basepoints because $L^2=0$. In particular, $|L_f|$ is also a basepoint free pencil, so the restriction $\Hm^0(R,L) \to \Hm^0(f,L_f)$ is an isomorphism.
\end{proof}
By Grauert's Theorem, $\pi_*(L)$ is a vector bundle on $\P^1$ of rank 2, and we have a double cover $\psi_L:R \to \P(\pi_*(L))$ given by the relative complete linear system.
\begin{prop} \label{prop:res-covers-hirzebruch}
The Hirzebruch surface $\P(\pi_*(L))$ is either $\F_0$, $\F_1$, or $\F_2$. The first two cases occur for $L$ ordinary (and $g$ even, odd, respectively), and the third case occurs for $L$ Weierstrass ($g$ is always even in that case).
\end{prop}
\begin{proof}
Let $n\geq 0$. Recall that each Hirzebruch surface $p \colon \F_n = \P(\O\oplus \O(n)) \to \P^1$ has Picard rank 2, with generators $h = p^* c_1(\O_{\P^1}(1))$ and $\zeta = c_1(\O_{\F_n}(1))$ the relative hyperplane class. These also generate the effective cone, and their intersection pairing is:
\[
h^2 =0 ,\quad \zeta^2 = -n,\quad h\cdot \zeta = 1.
\]
The cone of movable curves (dual to the effective cone in this case) is given by classes $a\zeta+bh$ such that $b/a>n$, $a>0$. If $\psi_L:R\to \F_n$ is the double cover, then by the Riemann-Hurwitz formula,
$K_R = \psi_L^*K_{\F_n} + \mathrm{Ram} = \psi_L^* K_{\F_n} + \frac{1}{2}\psi_L^*\Br$.
Since $-K_R = \psi_L^* h$, we deduce that $\Br = 4\zeta+(2n+2)h$. This is a divisor with slope $b/a=(n+1)/2$. For $n\geq 2$, the linear system of $\Br$ on $\F_n$ must have a fixed part. For $n=2$, the fixed part is $\zeta$, and for $n\geq 3$, the fixed part is $2\zeta$. This eliminates the possibility $n\geq 3$, because a double cover of $\F_n$ with branch divisor containing $2\zeta$ is a non-normal surface.

Next, we determine the divisor class $M$ on $\F_n$ such that $\psi_L^*(M)=L$. It must be a section class, since $L$ is a bisection of $R$, and $2g = L^2 = 2M^2$, so $M^2=g$. If we set $M = \zeta+kh$, then solving this equation on $\F_n$ gives
$k = \frac{1}{2}(g+n).$
Since $k$ must be an integer, $g$ and $n$ must have the same parity. Recall that a line bundle $L$ on $R$ is Weierstrass if and only if $|L|$ contains a member of the form $2s+\frac{1}{2}(g+2)f$, where $s$ is a section of $\pi:R\to \P^1$. This is possible iff $g$ is even and $n=2$.
\end{proof}
The above proposition leads us to two important constructions. The first is of the curve $\Gamma\subset R$ as required in Lemma \ref{lem:vanishingtheorem}, and the second is of the moduli stacks that parametrize families of pairs $(R,L)$, equipped with their universal families of Severi curves.
\begin{lem}\label{lem:gamma}
    Given an irreducible rational bisection curve $B\in |L|$ on $R$, there exists a reduced divisor $\Gamma\subset R$ in the class $B+mF$ for some $m\gg 0$, such that $\Gamma$ passes through each node point of $B$ with multiplicity 3, and $\Gamma$ has an $A_d$ singularity at each $A_{d-1}$ singularity of $B$, for $d\geq 3$, with the same preferred direction.
\end{lem}
\begin{proof}
The rational bisection $B$ is equal to $\psi_L^*(C)$, where $C$ is a smooth curve on $\F_n$ in class $\zeta+kh$. Each $A_{d-1}$ singularity of $B$ lies over a point of $C\cap \Br$ with intersection multiplicity $d$, and both curves are smooth. To construct $\Gamma$, we take $\psi_L^*(C')+\sum f_i$, where $f_i$ are fiber curves passing through the nodes of $B$, and $C'$ is a smooth curve on $\F_n$ in class $\zeta+k'h$ for $k'\gg 0$. To satisfy the desired property, we ask that $C'$ meet $B$ at each point of $C\cap \Br$ with intersection multiplicity $d+1$. All this can be arranged because $C'$ and $C$ are birational to the graphs of polynomial morphisms $\mathbb A^1 \to \mathbb A^1$. In the case of $C'$, the degree of the morphism is arbitrarily large, so any local intersection multiplicity with $\Br$ can be arranged by Lagrange interpolation.
\end{proof}

We will now describe a countable collection of 8-dimensional reduced and irreducible moduli stacks $\MRES_g^{\mathrm{ord}}$ and $\MRES_g^{\mathrm{Wei}}$, where the latter only exists for $g>0$ even, which represent the moduli functors for pairs $(R,L)$ with the desired discrete invariants fxed. We begin with an informal discussion of quotient stacks.

The automorphism group of $\F_n$ for $n>0$ can be described in terms of base and fiber:
$$\Aut(\F_n) \simeq \PGL(\O_{\P^1}\oplus \O_{\P^1}(n)) \rtimes \PGL(2),$$
and hence has dimension $n+5$. For $n=0$, the automorphism group has two connected components of dimension 6, cosets of $\PGL(2)\times \PGL(2)$ under the swapping of the two factors.

For $g$ even, in the ordinary and Weierstrass cases respectively, we will use
\begin{align*}
    |4\zeta+2h|_{\sm} &\subset \P \Hm^0(\F_0,4\zeta+2h),\\
    |4\zeta+6h|_{\sm} &\subset \P \Hm^0(\F_2,4\zeta+6h) \simeq \P \Hm^0(\F_2,3\zeta+6h),
\end{align*}
to denote the open loci of smooth curves in each complete linear system.
The actions of $\Aut(\F_0)$ and $\Aut(\F_2)$ on these respective schemes have finite stabilizers, since smooth curves of genus $3$, resp. genus 4, have finite automorphism groups. The quotient stacks
\begin{align*}
    [|4\zeta+2h|_{\sm} / \Aut(\F_0)],\\
    [|4\zeta+6h|_{\sm}  / \Aut(\F_2)],
\end{align*}
are Deligne-Mumford stacks, by \cite{alper-hall-rydh-luna}, and they have the same coarse spaces as $\MRES_g^{\ord}$ and $\MRES_g^{\Wei}$, respectively, as we will see. They are reduced and irreducible of dimension 8.

For $g$ odd, we similarly use $|4\zeta+4h|_{\sm}\subset\P \Hm^0(\F_1,4\zeta+4h)$ to denote the open locus of smooth (genus 3) curves in the complete linear system, and then observe that
$$ [|4\zeta+4h|_{\sm} / \Aut(\F_1)]$$
has the same coarse space as $\MRES_g^{\ord}$, and it is again integral of dimension 8.
\begin{remark}
For $n>0$, the automorphism group $\Aut(\F_n)$ is non-reductive because it contains unipotent matrices. We do not claim that the coarse spaces of these quotient stacks are quasi-projective. Nonetheless, they are finite type algebraic spaces.
\end{remark}

By Proposition \ref{prop:res-covers-hirzebruch}, every \( (R,L) \) with irreducible fibers arises as the double cover of some Hirzebruch surface \( \F_n \) branched along a smooth divisor of class \( \Brn \defeq \O_{\F_n}(4) \otimes p^* \O_{\P^1}(2n + 2)\), with \( L \) the pullback of a line bundle \( M \) on \( \F_n \), where \( n \), \( \Brn \), and \( M \) only depend on \( g(L) \), and whether \( L \) is ordinary or Weierstrass.
This \emph{almost} allows us to conclude that the moduli stacks for pairs $(R,L)$ are simply the quotient stacks \( [\abs{\Brn}_\sm / \Aut(\F_n)]\) mentioned above.

This does not quite work because the universal curve over \( \abs{\Brn}_\sm \) is a divisor on \( \abs{\Brn}_\sm \times \F_n \) whose class is not 2-divisible in the Picard group.
Hence, there is no double cover branched along it, and thus there is no universal family of \( (R,L) \) over \( \abs{\Brn}_\sm \).
This is solved by first passing to a \emph{root gerbe}; see \cite[Example 3.9.21]{alper} for generalities on root gerbes.
\begin{defn}
  For \( \star \in \set{\ord,\Wei} \), we let \( \wt{\MRES}_g^\star \) be the root gerbe of \( \abs{\Brn}_\sm \) with respect to the line bundle \( \O_{\abs{\Brn}_\sm}(1) \), where \( g \) is required to be even when \( \star = \Wei \), and where \( n\in \{0,1,2\} \) is determined as in Proposition~\ref{prop:res-covers-hirzebruch}: \( n = 2 \) if \( \star = \Wei \), and if \( \star = \ord \) then \( n \) is \( 0 \) or \( 1 \) when \( g \) is even or odd, respectively.
\end{defn}

That is, a morphism \( T \to \wt{\MRES}_g^\star \) from a test scheme \( T \) is a triple \( (f,N,\theta) \), where \( f \) is a morphism \( f \colon T \to \abs{\Brn}_\sm \), \( N \) a line bundle on \( T \), and \( \theta \) an isomorphism \( N^{\otimes 2} \toi f^* \O(1) \).

\begin{defn} Let \( \MRES_g^\star \) be the stack of pairs \( (R,L) \) with \( R \) a smooth rational elliptic surface -- \emph{without a specified section} -- and \( L \) a genus \( g \) bisection line bundle on \( R \) of type \( \star \).
\end{defn}
That is, a morphism \( T \to \MRES_g^{\star} \) from a test scheme \( T \) is by definition a pair \( (p \colon \cR \to T, \mathcal{L}) \) consisting of a proper smooth family of rational elliptic surfaces, without a specified section, over \( T \) and a section \( \mathcal{L} \) of the relative Picard scheme \( \Pic_{\cR/T} \to T \),  whose restriction to \( \Pic(R_0) \) for each surface \( R_0 \) in \( \cR \) is a genus \( g \) bisection class of type \( \star \).
This pseudo-functor defines a stack over \( \Sch_{\C} \) in an evident manner.

\begin{prop}
There is a natural morphism \( q \colon \wt{\MRES}_g^\star \to \MRES_g^\star \).
\end{prop}
\begin{proof}Recall that the universal property of \( \abs{\Brn}_\sm \) implies that for any scheme $T$, morphisms \( f \colon T \to \abs{\Brn}_\sm  \) are in bijection with effective Cartier divisors \( D \subset T \times \F_n \) whose restriction to each fibre \( \set{t_0} \times \F_n \) is of class \( \Brn \).
Moreover, under this bijection, we have
$$ \O(D) \cong f^* \O(1) \boxtimes \Brn,
$$
and
$$f^* \O(1) \cong N_D \defeq (\pr_1)_*(\O(D) \otimes \pr_{2}^*(\Brn \I)).$$
Hence, a morphism \( T \to \wt{\MRES}_g^\star \) can be regarded as a triple \( (D,N,\theta) \) with \( D \subset T \times \F_n \) a divisor as above, \( N \) a line bundle on \( T \), and an isomorphism \( \theta \colon N^{\otimes 2} \to N_D \).

The line bundle \( \Brn = \O_{\F_n}(4) \otimes p^* \O_{\P^1}(2n + 2) \) has a square root 
\[ \Brn^{1/2} \defeq \O_{\F_n}(2) \otimes p^* \O_{\P^1}(n + 1).\]
Thus, given a triple \( (D,N,\theta) \) as above, since \( \O(D) \cong N_D \boxtimes \Brn \), the square root \( \theta \colon N^{\otimes 2} \to N_D \) gives us a square root \( N \boxtimes \Brn^{1/2} \) of \( \O(D) \).
We recall the cyclic cover construction from \cite[Proposition~4.1.6]{lazarsfeld-positivity_1} that produces from this data a double cover of \( T \times \F_n \) branched over \( D \), which we denote by \( \psi_{D,\theta} \colon \cR_{D,\theta} \to T \times \F_n \).

Now, \( \cR_{D,\theta} \) is a smooth family of rational elliptic surfaces over \( T \).
Moreover, as in the proof of Proposition~\ref{prop:res-covers-hirzebruch} above, the pullback \( \psi_{D,\theta}^* \pr_2^* \O_{\F_n}(\zeta+\frac 12 (g + n)h) \) is a line bundle on \( \cR_{D,\theta} \) of genus \( g \) and type \( \star \), which gives a class in \( \Pic(\cR_{D,\theta}) \) and hence a section \( \cL_{D,\theta} \) of \( \Pic_{\cR_{D,\theta}/T} \). We define
$$q \colon \wt{\MRES}_g^\star \to \MRES_g^\star$$
on objects by taking a morphism from a test scheme \( T \to \wt{\MRES}_g^\star \) given by a triple \( (D,N,\theta) \) to the morphism \( T \to \MRES_g^\star \) given by the pair \( (\cR_{D,\theta},\cL_{D,\theta}) \).

The cyclic cover construction is functorial in \( N \) in the sense that, given \( \theta' \colon (N')^{\otimes 2} \toi N_D \), any isomorphism \( \alpha \colon N \to N' \) with \( \theta' \alpha^{\otimes 2} = \theta \colon N^{\otimes 2} \to N_D \) induces an isomorphism \( \cR_{D,\theta} \toi \cR_{D',\theta'} \) over \( T \times \F_n \), which thus takes \( \cL_{D,\theta} \) to \( \cL_{D',\theta'} \).
This shows that a 2-cell between given morphisms \( T \rightrightarrows \wt{\MRES}_g^\star \) induces an isomorphism between their images under \( q \).
One routinely verifies that this prescription is compatible with composition and with base change along morphisms \( T' \to T \) -- i.e., that $q$ defines a morphism of stacks.
\end{proof}
Recall the action of \( \Aut(\F_n) \) on \( \abs{\Brn}_\sm \), and denote by $\overline{q}$ the stack quotient morphism.
\begin{prop}
  There is up to isomorphism a unique morphism \( \Phi \) making the square
  \begin{equation}\label{eq:gerbe-square}
    \begin{tikzcd}
      \wt{\MRES}_g^\star \ar[r, "u"] \ar[d, "q"] &
      \abs{\Brn}_\sm \ar[d, "\overline{q}"] \\
      \MRES_g^\star \ar[r, "\Phi"] &
      \mathrm{[}\abs{\Brn}_\sm / \Aut(\F_n)\mathrm{]}
    \end{tikzcd}
  \end{equation}
  2-commute, and the resulting square is 2-cartesian.
\end{prop}
\begin{proof}
To define \( \Phi \), we need to assign to each morphism \( T \to \MRES_g^{\star} \) from a test scheme \( T \) a morphism \( T \to \mathrm{[}\abs{\Br_n}_\sm / \Aut(\F_n)\mathrm{]} \).
A morphism \( T \to \MRES_g^\star \) is by definition a pair \( (\cR,\cL) \).
We may assume that the family of Hirzebruch surfaces associated to \( \cR \) is trivial, i.e., isomorphic to \( T \times \F_n \), since every \( (\cR,\cL) \) is locally of this form, so by the descent axiom for objects of a stack, it suffices to define \( \Phi \) on test morphisms of this form.
This assumption is precisely that \( (\cR,\cL) \) is in the image of \( q \), i.e., that it is of the form \( (\cR_{D,\theta},\cL_{D,\theta}) \) for some \( D \subset T \times \F_n \) and some \( \theta \colon N^{\otimes 2} \toi N_D \).
The commutativity of the square then forces us take \( \Phi(\cR,\cL) \) to be \( \overline{q}(u(D,N,\theta)) \).
Similarly, the definition of \( \Phi \) on a 2-cell between morphisms \( T \toto \MRES_g^\star \) is determined by the commutativity of the square, and a straightforward check reveals that the resulting \( \Phi \) is a morphism of stacks.

Let us see that the resulting square is \( 2 \)-cartesian.
Given \( t \colon T \to \MRES_g^\star \) such that \( \Phi \circ t \) lifts to \( \abs{\Brn}_\sm \), it follows that the pair \( (\cR,\cL) \) is such that the associated family of Hirzebruch surfaces is trivial, and hence that \( t \) lifts to \( \wt{\MRES}_g^\star \).
Similarly, given \( t_1,t_2 \colon \wt{\MRES_g^\star} \) and 2-cells \( q \circ t_1 \toi q \circ t_2 \) and \( u \circ t_1 \toi u \circ t_2 \) inducing the same 2-cell between \( T \toto \abs{\Br_n}_\sm / \Aut(\F_n) \), we need to show that there is a unique 2-cell \( t_1 \toi t_2 \) giving rise to these.
But this is just the fact that, given \( (D,M,\theta) \) and \( (D,M',\theta') \), each isomorphism \( \cR_{D,\theta} \toi \cR_{D,\theta'} \) is induced by a unique \( M \toi M' \).
\end{proof}

\begin{cor} \label{cor:mres-is-smooth}
  All four stacks appearing in \eqref{eq:gerbe-square} are smooth, integral Deligne-Mumford stacks, and \( \MRES_g^\star \) is 8-dimensional.
\end{cor}

\begin{proof}
  As an open subscheme of projective space, \( \abs{\Brn}_\sm \) is smooth and integral.
  The quotient morphism \( \overline{q} \) is smooth and surjective, so the quotient stack \( [\abs{\Brn}_\sm / \Aut(\F_n)] \) is smooth and integral by \cite[04YH]{stacks-project}.
  The desired properties for $\wt{\MRES_g}^\star$ and $\MRES_g^\star$, follow from the fact that $u$ is a $\mu_2$-gerbe, and so $\Phi$ is also a $\mu_2$-gerbe by \cite[06QF]{stacks-project}. The dimension of \( \MRES_g^\star \) can be computed from the dimension of the quotient stack:
  \[ 8 = \dim|\Brn| - \dim \Aut(\F_n).  \]
\end{proof}

\begin{cor}\label{cor:severi-invariance}
    The arithmetic genus of $V(R,L)$ and the geometric genus of $V(R,L)$ for general $(R,L)$, depends only on $g(L)$, and whether $L$ is ordinary or Weierstrass.
\end{cor}
\begin{proof}
  This follows since all such \( V(R,L) \) occur in a single family $\mathcal{V}^{\star}_g \to \MRES^{\star}_g$ over the integral stack \( \MRES_g^\star \).
  Indeed, we have universal family of pairs \( (R,L) \) over \( \MRES_g^\star \), and by \cite{harris-diaz-ideals}, Severi varieties can be constructed in families over a reduced base.
\end{proof}

\section{Noether-Lefschetz Theory}
For the eventual genus bound on $V(R,L)$, we now introduce the machinery of (quasi)-modular forms valued in the cohomology of moduli spaces of polarized K3 surfaces. In \S\S\ref{subsec:lattice-polarized}-\ref{subsec:nl-theory}, we review the general theory, and in \S\ref{subsec:families-of-k3s}, we apply it to a specific family of K3 surfaces.

\subsection{Lattice-polarized K3 surfaces}\label{subsec:lattice-polarized}
Let $\Lambda_{\K3} = U^{\oplus 3}\oplus \E8(-1)^{\oplus 2}$ be the middle cohomology lattice of an arbitrary K3 surface. For our purposes, we will consider only special K3 surfaces, which admit an elliptic fibration structure, or even further, come from a rational elliptic surface via base change. Let $M$ be a lattice of signature $(1,l)$ equipped with an isometric embedding into $\Lambda_{\K3}$ and a {\it very irrational} vector\footnote{A vector $v\in M\otimes \R$ is very irrational if it is not contained in $M'\otimes \R$ for any sublattice $M'\subset M$ of smaller rank.} $v\in M\otimes \R$ with $v^2>0$.

\begin{defn}
An $M$-quasi-polarized K3 surface is a pair $(S,\iota)$, where $S$ is a smooth K3 surface, and $\iota:M \hookrightarrow \NS(S)$ is a primitive isometric embedding such that $\iota(v)$ is a nef and big $\R$-divisor class.
\end{defn}
This definition differs slightly from the original one found in \cite{dolgachev-lattice-polarized}; it appears in \cite{alexeev-engel}, where the authors corrected the original definition to make the main theorems of \cite{dolgachev-lattice-polarized} true.
There is a period map that associates to an $M$-quasi-polarized K3 surface the weight 2 polarized Hodge structure on $M^\perp\subset \Lambda_{\K3}$. This map is an isomorphism --- see Theorem~\ref{dolgachev-torelli} below.

\begin{defn}
The orthogonal complement $M^\perp\subset \Lambda_{\K3}$ is a lattice of signature $(2,19-l)$. The period domain of $M^\perp$ is defined by
$$\D(M^{\perp}) \defeq\{\omega \in \P(M^\perp\otimes \C): (\omega,\omega)=0,\, (\omega,\overline{\omega})>0 \},$$
which has two connected components interchanged by complex conjugation.
\end{defn}
\begin{defn}
The stable orthogonal group $\tilde{O}(M^\perp)$ is the set of lattice automorphisms which act trivially on the discriminant group $(M^{\perp})^{\vee}/M^\perp$. 
\end{defn}
This subgroup consists precisely of those automorphisms which can be lifted to automorphisms of $\Lambda_{\K3}$ acting trivially on $M$.

The arithmetic quotient $\D(M^{\perp})/\tilde{O}(M^\perp)$ has an analytic structure isomorphic to that of a quasi-projective variety. In fact, it is a Shimura variety of orthogonal type for the group $O(2,19-l)$, and this remains true when $\tilde{O}(M^\perp)$ is replaced with any arithmetic subgroup $\Gamma\subset O(2,19-l)$.
\begin{thm}\label{dolgachev-torelli}%
\cite[after Proposition~3.3]{dolgachev-lattice-polarized}
Let $\mathcal M^{M}_{\K3}$ be the coarse moduli space of $M$-quasi-polarized K3 surfaces. The period map
$$\mathcal M^{M}_{\K3} \to \D(M^\perp)/\tilde{O}(M^\perp)$$
is an isomorphism.
\end{thm}
Now suppose that $M'\subset M$ is a sublattice of signature $(1,l')$ with $l'<l$. (Our case of interest will be $U\subset U\oplus \E8(-2)$.) Then we have natural inclusions of the data defined above.
\begin{align*}
    M^\perp&\subset M'^\perp\\
    \D(M^\perp) &\subset \D(M'^\perp)
\end{align*}
If $\Gamma_M\subset \tilde{O}(M'^\perp)$ is the stabilizer of the sublattice $M^\perp\subset M'^\perp$, then we have a finite morphism of Shimura varieties
\begin{equation}\label{finitemorphism}
  \D(M^\perp)/\Gamma_M \to \D(M'^\perp) / \tilde{O}(M'^\perp).
\end{equation}
Any orthogonal type Shimura variety $\D(M^\perp)/\Gamma$ admits a compactification due to Satake and Baily-Borel. To construct it, the line bundle $\O_{\D(M^\perp)}(-1)$ descends to an orbi-bundle $\lambda$ on $\D(M^\perp)/\Gamma$ called the Hodge bundle, and we take
$$(\D(M^\perp)/\Gamma)^*  = \Proj\left( \bigoplus_{n\geq 0} H^0(\D(M^\perp)/\Gamma, \lambda^{\otimes n}) \right).$$
\begin{prop}\label{bbext}
The finite morphism \eqref{finitemorphism} of Shimura varieties extends to the respective Satake-Baily-Borel compactifications.
\end{prop}
\begin{proof}
  Since $\O_{\D(M'^\perp)}(-1)$ restricts to $\O_{\D(M^\perp)}(-1)$ on $\D(M^\perp)$, the Hodge line bundle $\lambda$ on $\D(M'^\perp)/\tilde{O}(M'^\perp)$ restricts to Hodge line bundle on $\D(M^\perp)/\Gamma_M$. This induces a morphism of graded rings, and hence a morphism of projective varieties
\[
(\D(M^\perp)/\Gamma_M)^* \to (\D(M'^\perp) / \tilde{O}(M'^\perp))^*.
\qedhere
\]
\end{proof}
Topologically, the boundary of $(\D(M^\perp)/\Gamma)^*$ is a union of modular curves indexed by $\Gamma$-orbits of rank 2 primitive isotropic sublattices $J\subset M^\perp$, and points indexed by $\Gamma$-orbits of rank 1 primitive isotropic sublattices $I\subset M^\perp$. The latter appear as cusps of the former when $I\subset J$, but there may also be isolated boundary points. For each $J$, there are restriction maps \begin{align*}
    p&\colon \Stab_\Gamma(J) \to GL(J)\simeq GL_2(\Z),\\
    q&\colon \Stab_\Gamma(J) \to O(J^\perp/J).
\end{align*}
The open modular curve indexed by (the $\Gamma$-orbit of) $J$ is given by
$$Y_J\defeq\left(\P^1(J_\C) \smallsetminus \P^1(J_\R)\right)/p(\Stab_\Gamma(J)).$$
While the Satake-Baily-Borel compactification is badly singular at the boundary, one can obtain smooth compactifications obtained by blowing up. Mumford et al.\ \cite{mumford-smooth-compactifications} constructed a toroidal compactification for each $\D(M^\perp)/\Gamma$ depending on a choice of $\Gamma$-admissible fan $\Sigma$. All such choices agree if we consider only the partial compactification over the open modular curves indexed by $J$. We will refer to this partial compactification as $(\D (M^\perp)/\Gamma)^{\II}$.

The boundary of this partial compactification also has a concrete description. It consists of a union of fiber bundles over the open modular curves, where the fiber $\tau\in Y_J$ in the modular curve is :
$$(E_\tau \otimes_\Z (J^\perp/J))/q(\Stab_\Gamma(J)).$$
Here, $E_\tau$ is the elliptic curve with period $\tau$. Note that since $J^\perp/J$ is negative definite, the group $q(\Stab_\Gamma(J))\subset O(J^\perp/J)$ is finite.

The partial toroidal compactification is also functorial under isometric embeddings $M'\subset M$ as follows:
\begin{prop}\label{prop:partial-toroidal-functorial}
The finite morphism \eqref{finitemorphism} of Shimura varieties extends to the respective partial toroidal compactifications of Type II.
\end{prop}
\begin{proof}
  The Type II partial toroidal compactification is the blow up of the boundary in the partial Baily-Borel compactification where only the open strata $Y_J$ are included. The extension of Proposition~\ref{bbext} amounts to coverings of open modular curves at the boundary. The desired extension here then follows from the universal property of blow ups.
\end{proof}

\subsection{Noether-Lefschetz divisors and modularity}\label{subsec:nl-theory}

Consider a flat family $\mathcal{X}\to T$ whose general fiber is a polarized K3 surface. We wish to count members of this family with jumping Picard rank.
This can be done by intersection theory with a collection of divisors in the appropriate period space, corresponding to a choice of polarizing lattice $M\subset\Lambda_{\K3}$ as in \S\ref{subsec:lattice-polarized}.

The theory is simpler in the case where $M$ is unimodular, and we will restrict ourselves to this case.
Specifically, we will assume that we have a map $\mathcal{X}\to T\times\P^1$ with a section $T\times\P^1\to\mathcal{X}$ giving the elliptic fibration structure on $X_t$ for each $b\in T$.
Thus for any elliptic K3 surface $X_t\to \P^1$ in our family, the fiber and zero section classes span a sublattice of $\NS(X)$ isomorphic to $U$, the rank 2 hyperbolic lattice, with Gram matrix
\[
  \begin{pmatrix} 0 & 1 \\ 1 & -2 \end{pmatrix}.
\]
Hence $\mathcal X \to T$ may be viewed as a family of $U$-polarized K3 surfaces.

\begin{remark}
We will use $U$ as the polarizing lattice in Section \ref{subsec:families-of-k3s} instead of a potentially larger available Neron-Severi group because $U$ is unimodular. This simplifies the possibilities for the Noether-Lefschetz series, at the expense of creating some excess intersection contributions. Furthermore, there is a unique isometric embedding $U\hookrightarrow \Lambda_{K3}$ up to integral isometries of $\Lambda_{K3}$, by \cite{james-witts-theorem}.
\end{remark}

\begin{defn}
Let $\Lambda = U^{\oplus 2}\oplus \E8(-1)^{\oplus 2}$, the orthogonal complement of the polarizing $U$ inside $\Lambda_{\K3}$. Let $\D(\Lambda)/\Gamma$ be the period space for $U$-quasi-polarized K3 surfaces as in \S\ref{subsec:lattice-polarized}.
\end{defn}
\begin{defn}
  For each positive integer $n$, the Noether-Lefschetz locus $\NL_n$ is given by
  \[\NL_n = \left(\bigcup_{\substack{v\in \Lambda \\ (v,v)=-2n}} v^\perp\right)/\Gamma\subset \D(\Lambda)/\Gamma.\]
\end{defn}

By Theorem~\ref{dolgachev-torelli}, our family $\mathcal{X}\to T$ of K3 surfaces has an associated (rational) lattice-polarized period map
\[j\colon T\dasharrow \mathbb  D(\Lambda)/\Gamma\]
defined for $t\in T$ such that $X_t$ has at worst ADE singularities. By the Lefschetz (1,1)-theorem, the image of $j$ meets the Noether-Lefschetz loci exactly when the Picard rank of $X_t$ is $\geq 3$.
\begin{remark}\label{rmk:brieskorn}
By Brieskorn's simultaneous resolution, an ADE singular K3 has the same period point as its minimal resolution, which registers a Picard rank jump.
\end{remark}

The $\NL_n$ are finite unions of codimension 1 subvarieties. Results of Borcherds \cite{borcherds-gross-kohnh-zagier-in-higher-dim} and Kudla-Millson \cite{kudla-milson-intersection-numbers-and-fourier} show that they are the coefficients of a cycle-valued modular form.
\begin{thm}\label{borch}
Let $\lambda$ be the class of the Hodge bundle on $ \D(\Lambda)/\Gamma$. Then
\[\Phi(q) = -\lambda + \sum_{n\geq 1} [\NL_n]q^n \in \Mod(\SL_2(\Z),10)\otimes_\Q \Hm^2(\D(\Lambda)/\Gamma,\Q).\]
\end{thm}
\noindent
Now the space of classical modular forms of weight 10 is 1-dimensional:
\[\Mod(\SL_2(\Z),10) = \Q \eis_4 \eis_6.\]
To restate Theorem \ref{borch} more concretely, if we intersect a fixed element $\alpha\in \Hm_2(\D(\Lambda)/\Gamma,\Q)$ with each coefficient of the power series $\Phi(q)$, the result will be
\[-(\lambda\cdot \alpha)\eis_4 \eis_6.\]
Because our period map $j\colon T\dasharrow\D(\Lambda)/\Gamma$ is only rational, we do not in general get a pushforward class $\alpha\in \Hm_2(\D(\Lambda)/\Gamma,\Q)$.
However, following \cite{greer-quasimod-from-mixed}, we can extend the period map $j$ to a regular map $j^{\II}$ to the smooth partial toroidal compactification $(\D(\Lambda)/\Gamma)^{\II}$ of the arithmetic quotient from \S\ref{subsec:lattice-polarized}, for families $\mathcal{X}\to T$ with singular fibers of Type II.
We will denote by $\overline{\NL}_n\subset(\mathbb{D}/\Gamma)^{\II}$ the closures of $\NL_n\subset\mathbb{D}/\Gamma$.
By Theorem 23 in \cite{greer-quasimod-from-mixed}, we have the following generalization of Theorem \ref{borch}:
\begin{thm}\label{thm:quasimod}
Let $\lambda$ be the pullback of the Hodge bundle to $ (\D(\Lambda)/\Gamma)^{\II}$. Then
\[\Phi^{\II}(q) = -\lambda + \sum_{n\geq 1} [\overline{\NL}_n]q^n \in \QMod(\SL_2(\Z),10)\otimes_\Q \Hm^2((\D(\Lambda)/\Gamma)^{\II},\Q).\]
\end{thm}
In \S\ref{subsec:families-of-k3s}, we will compute the resulting quasi-modular form $j^{\II}_*[T]\cdot\Phi^{\II}(q)$ for a certain family $\mathcal{X}\to T$ of interest from the perspective of Severi curves.

\subsection{Families of K3 surfaces from a RES}\label{subsec:families-of-k3s}
Let $\pi\colon R\to\P^1$ be a general rational elliptic surface, and now fix a general double cover $u:\P^1\to\P^1$
with branch divisor $p_1+p_2$ in $\P^1$, and consider the pullback square:
\begin{equation}\label{square}
\begin{tikzcd}
  X \ar[r, "u'"]\ar[d, "\pi'"'] & R \ar[d, "\pi"] \\
  \P^1 \ar[r, "u"] & \P^1
\end{tikzcd}
\end{equation}
\noindent
The vertical morphisms are elliptic fibrations, and $X$ is an elliptic K3 surface. The top morphism $u'$ is a double cover branched along the pair of elliptic fibers $\pi^{-1}(p_1)+\pi^{-1}(p_2)$. Let $B\subset R$ be an irreducible rational bisection. Since $\pi|_B:B\to \P^1$ is a double cover, it has some branch divisor $q_1+q_2$ in $\P^1$. If $q_1+q_2$ is disjoint from $p_1+p_2$, then $(u')^{-1}(B)$ will be a curve of geometric genus 1 on $X$. 

This perspective provides a shorter, independent proof of the fact that the Severi varieties we consider are 1-dimensional, which is part of Theorem \ref{thm:reducedseveri-with-proof}.
\begin{prop}\label{expdimseveri}
Let $B\subset R$ be an irreducible rational bisection, and set $L=\O_R(B)$. Then the Severi variety $V(L)\subset |L|$ is 1-dimensional.
\end{prop}
\begin{proof}
  Geometric genus 1 curves on a K3 surface vary in 1-parameter families, since no K3 surface over $\C$ is uniruled. Any family of irreducible rational bisections in $R$ induces a family of genus 1 curves on a K3 surface by the construction above. The same argument shows that rational multisections of degree $d$ vary in a family of dimension $d-1$.
\end{proof}
\noindent
Now suppose that $p_1+p_2=q_1+q_2$. Then $(u')^{-1}(B)$ will be a reducible curve with two rational components, each a section of $\pi'$, conjugate under the involution of $X$ associated to the double cover $u'$.
\begin{prop}\label{prop:galois}
Let $\Sigma \subset X$ be a section of $\pi'$. Then either (i) $\Sigma$ is the base change of a section of $\pi$, or (ii) $\Sigma$ is a component of $(u')^{-1}(B)$, where $B$ is a bisection of $\pi$ with the same branch points as $u$.
\end{prop}
\begin{proof}
Consider $u'|_\Sigma : \Sigma \to R$. Since the pullback square \eqref{square} is commutative,
\[\pi\circ u'|_\Sigma: \Sigma \to \P^1\]
is a double cover. This implies that either (i) $u'|_\Sigma$ is a double cover, or (ii) $u'|_\Sigma$ is generically 1-1 and its image is a bisection of $\pi$. In the first case, $u'(\Sigma)$ is a section of $\pi$, and in the second case, $u'(\Sigma)$ is a bisection of $\pi$ with the same branch points as $u$.
\end{proof}
\begin{prop}
The surface $X$ is smooth iff the branch divisor $p_1+p_2$ is disjoint from the critical values of $\pi$. Otherwise, $X$ has an ordinary double point located at the node of a singular fiber over the critical value in question.
\end{prop}
\begin{proof}
This follows from a local computation. If $b(x,y)=0$ cuts out a smooth curve, then $z^2 = b(x,y)$ cuts out a smooth surface. If $b(x,y)=0$ cuts out a nodal curve, then $z^2 = b(x,y)$ cuts out a nodal surface.
\end{proof}
\noindent
To relate the study of bisections of $R$ to Noether-Lefschetz theory for K3s, we construct a test curve $T$ in the moduli space of elliptic K3 surfaces by considering a 1-parameter family of double covers $u_t:\P^1 \to \P^1$, and pulling back the fibration $\pi:R \to \P^1$.

Let $Y$ be the double cover of $\P^1\times \P^1$ branched along a general curve of bidegree $(2,2)$. % Reference: Lazarsfeld, Proposition~4.1.6
The family version of the pullback square \eqref{square} is the following Cartesian diagram:
\begin{equation}\label{eq:family-of-k3-diag}
  \begin{tikzcd}
    \mathcal X \ar[d, "h"']\ar[r, "\bar\rho"] & R\times \P^1 \ar[d, "{\pi\times\id}"] \\
    Y \ar[r, "\rho"] & \P^1\times \P^1 \ar[d, "\pr_2"] \\
    & \P^1&[-40pt]=T.
  \end{tikzcd}
\end{equation}
The resulting morphism $f \defeq \pr_2 \circ \rho \circ h :\mathcal X \to T$ gives a family of K3 surfaces which is generically smooth, but has some nodal and even reducible fibers. Indeed, the $(2,2)$-curve is ramified at 4 points over $T$, and these correspond to reducible curves $\P^1\cup_x\P^1$ in the family $Y\to T$.
The corresponding fibers of $\mathcal{X} \to T$ are of the form $R\cup_ER$ -- two rational elliptic surfaces glued along a smooth elliptic fiber.
\begin{prop}\label{prop:conic}
The family of double covers of $\P^1$ induces a map
\[T \to \Sym^2(\P^1).\]
The image of $T$ under the isomorphism $\Sym^2(\P^1) \simeq \P^2$ is a conic.
\end{prop}
\begin{proof}
The diagonal $\Delta\subset \Sym^2(\P^1)$ is a conic given by the discriminant form $b^2-4ac$. Since the image of $T$ meets $\Delta$ in 4 points, it must also be a conic by B\'{e}zout's Theorem. In fact, by a dimension count, a general choice of $(2,2)$-curve yields a general conic in $\Sym^2(\P^1)$.
\end{proof}
\noindent
For each bisection line bundle $L$ on $R$ we have a Severi curve $V(L)$ along with a branch locus map $V(L) \dasharrow \Sym^2(\P^1)$. After composing with the normalization of $V(L)$, this extends to a regular map:
\[f_L: \widetilde{V}(L) \to \Sym^2(\P^1).\]
By Proposition \ref{prop:galois}, each intersection point of $T$ with $\widetilde{V}(L)$ in $\Sym^2(\P^1)$ corresponds to a K3 surface in the family $\mathcal X \to T$ with two extra sections that are Galois conjugate. This observation will allow us to use the Noether-Lefschetz numbers of $T$ to study the degrees of the images of the branch morphisms $f_L$.

\begin{prop}\label{prop:hodgedegree}
  The Hodge line bundle \( f_* (\omega_{\mathcal{X}/T}) \) of the K3 family $f \colon \mathcal X \to T$ is isomorphic to $\O_{\P^1}(1)$.
\end{prop}
\begin{proof}
For simplicity, we write $\O(p,q)$ for the pullback to any of $\mathcal{X}$, $Y$, or $R\times\P^1$ of the sheaf $\O(p,q)=\O(p)\boxtimes\O(q)$ on $\P^1\times\P^1$. Using the Hurwitz formula, $\omega_Y\simeq\rho^*\omega_{\P^1\times\P^1}\otimes\O(1,1)\cong\O(-1,-1)$. By repeated application of the projection formula, we have:
\begin{align*}
    f_*(\omega_{\mathcal X/T}) &\simeq f_*(\omega_{\mathcal X/Y}\otimes h^* \omega_{Y}\otimes f^*\omega_{T}^{-1}) \\
     &\simeq f_*(\bar\rho^*\pr_1^*\omega_{R/\P^1} \otimes h^*\omega_Y) \otimes \omega_T^{-1}\\
     &\simeq f_*(\O(1,0) \otimes \O(-1,-1)) \otimes \O(2)\\
     &\simeq {\pr_2}_*\rho_*h_*h^*(\O(0,-1)) \otimes \O(2)
\end{align*}

Next, we have
\begin{align*}
  \rho_*h_*h^*(\O(0,-1))&\simeq
  \rho_*(h_*\O_{\mathcal X}\otimes \O(0,-1))\\&\simeq
  \rho_*\O(0,-1)\\&\simeq
  \O(0,-1)\otimes\rho_*\O_Y\\&\simeq
  \O(0,-1)\otimes(\O\oplus\O(-1,-1))\\&\simeq
  \O(0,-1)\oplus\O(-1,-2)
\end{align*}
where the second to last isomorphism follows from \cite[Proposition~4.1.6]{lazarsfeld-positivity_1}. Hence,
\begin{align*}
f_*(\omega_{\mathcal X/T})&\simeq
{\pr_2}_*\Big(\O(0,-1)\oplus\O(-1,-2)\Big)\otimes\O(2)\\&\simeq
\Big(\Hm^0(\O)\otimes \O(-1)\oplus\Hm^0(-1)\otimes \O(-2)\Big)
\otimes\O(2)\\&\simeq
\O(1)\qedhere
\end{align*}

\end{proof}

The family $\mathcal X \to T$ has 4 simple normal crossing fibers $R\cup_E R$ that are of Type II in Kulikov's classification \cite{kulikov-k3s} and 24 nodal fibers which are of Type I. Hence, the associated rational period map $j\colon T\dasharrow \mathbb{D}(\Lambda)/\Gamma$ extends to a regular map $j^\II\colon T\to(\mathbb{D}(\Lambda)/\Gamma)^\II$, as described in \S\ref{subsec:nl-theory}.
Let us now compute the resulting quasi-modular form $j^\II_*[T]\cdot\Phi^\II(q)$.

\begin{prop}\label{prop:nl-series-formula}
  Let $\alpha = j^\II_*[T] \in \Hm_2((\D(\Lambda)/\Gamma)^\II,\Q)$. Then the power series
  \[ \varphi(q) \defeq \Phi^\II(q) \cdot \alpha = -\lambda\cdot \alpha + \sum_{n\geq 1} [\overline{\NL}_n]\cdot \alpha\, q^n \]
  is given by the following quasi-modular form of weight 10:
  \[\varphi(q) = -\frac{1}{3}\eis_2 \eis_4^2 - \frac{2}{3} \eis_4\eis_6.\]
\end{prop}
\begin{proof}
  Since the degenerations in our family $\mathcal X \to T$ are of Type II, we are in the situation of \cite{greer-quasimod-from-mixed}. Specifically, the holomorphic anomaly equation in Theorem~34 of \emph{loc. cit.} implies that $\varphi(q)$ is at most linear in $E_2$, and thus lies in the 2-dimensional subspace of $\QMod(\SL_2(\Z),10)$ spanned by $\eis_2\eis_4^2$ and $\eis_4\eis_6$. It is uniquely determined by the values $\lambda\cdot\alpha$ and $\Delta^\II\cdot\alpha$, where $\Delta^{\II}\subset (\D(\Lambda))^{\II}/\Gamma$ is the boundary divisor of the compactification:
  \[\varphi(q) = -(\lambda'\cdot \alpha) \eis_4\eis_6 - \frac{1}{8}(\Delta'\cdot \alpha)D(\eis_4^2).\]
  We have $\Delta^{\II}\cdot \alpha=4$, and $\lambda\cdot \alpha=1$ by Proposition \ref{prop:hodgedegree}. This gives the desired expression by the Ramanujan identity \eqref{eq:ramanujan} on page~\pageref{eq:ramanujan}.
\end{proof}

\section{Bounding the degree of $f_L$}\label{sec:bounding}
We will now relate the Noether-Lefschetz coefficients of the quasi-modular form $\varphi(q)$ computed in the last section to the genera of Severi curves $V(L)$ and thence prove Theorem~\ref{thm:genusbound}. Recall our main construction from \S\ref{subsec:families-of-k3s}: we fixed a general rational elliptic surface $R\to\P^1$, and we produced a family $\mathcal X \to T\simeq \P^1$ of K3 surfaces, each member of which is a double cover of $R$ branched at two elliptic fibers.

The key observation is that we now have two ways of counting the number of surfaces in the family $\mathcal{X}$ with jumping Picard rank.
On the one hand, these counts are by definition given by the Noether-Lefschetz numbers.
On the other hand, we know by Proposition~\ref{prop:galois} that $X_t$ has an extra section precisely when the corresponding double cover $u_t\colon\P^1\to\P^1$ has branch locus agreeing with that of some rational bisection $B\subset R$.
In other words, Picard jumping will occur precisely at the intersection points in $\Sym^2(\P^1)\cong\P^2$ of the maps $f_L\colon\wt V(L)\to\Sym^2(\P^1)$, for varying bisection line bundles $L$, with the fixed conic $T\to\Sym^2(\P^1)$.
We thus relate the degree of the plane curve $f_L\left(\wt{V}(L)\right)$, and hence its genus, with the Noether-Lefschetz numbers.

The precise relationship is complicated somewhat by the presence of certain correction terms in the Noether-Lefschetz numbers.
However, as we prove in \S\ref{subsec:geom-interp}, these terms are dominated asymptotically by the terms of interest.

Lastly, to deduce Theorem~\ref{thm:genusbound}, we show that the map $f_L\colon\wt V(L)\to\Sym^2(\P^1)$ is always birational onto its image, so that we can compute the genus of $V(L)$ in terms of the degree of the cycle-theoretic image of $f_L$.
This is carried out in \S\ref{subsec:mainthm}.

\subsection{Geometric interpretation of Noether-Lefschetz numbers}\label{subsec:geom-interp}
For reference, let us record the first few coefficients of the quasi-modular form $\varphi(q)$, which we determined in Proposition~\ref{prop:nl-series-formula}:
\[\varphi(q) = -1+24 q + 73512 q^2 + 3621216 q^3+\cdots\]
The constant coefficient $-1$ is the degree of the Hodge bundle of $\mathcal{X}\to T$, as computed in Proposition~\ref{prop:hodgedegree}.
The linear coefficient $24$ corresponds to the number of nodal surfaces in the family $\mathcal{X}\to T$.
This comes from the intersection of the conic $T\to\Sym^2(\P^1)$ with the 12 lines in $\Sym^2\P^2$ corresponding to the critical points of $R\to\P^1$.

The remainder of this section will be concerned with proving the following:
\begin{prop}\label{prop:degrees-are-modular}
  The quasi-modular form $\varphi(q)$ is given by
  \[
    -1+\sum_{n\ge0}\left(4\sum_L\deg\left({f_L}_*[\wt V(L)]\right)\right)q^{n+2}+\psi(q),
  \]
  where $\psi(q)=\psi_\ex(q)+\psi_\no(q)$ is a sum of two modular forms of lower weight, and $L$ in the sum ranges over bisection line bundles on $R$ of height $n$.
  We will determine $\psi_\ex$ and $\psi_\no$ explicitly.
\end{prop}
The coefficients $4\sum_L\deg\left(f_{L*}[\wt V(L)]\right)$ are what we expect for the numbers $[\overline{\NL}_n]\cdot\alpha$, given the relationship between the Noether-Lefschetz special surfaces in the family $\mathcal{X}$ (that is, those of Picard rank $>2$) and the Severi curves of bisection line bundles on $R$.
The two factors of 2 come from the fact that $T\subset \P^2$ is a conic, and that each intersection point gives two Galois-conjugate sections of the corresponding K3 surface.

The series $\psi(q)$ is a correction to this expectation coming from two sources, corresponding to the modular forms $\psi_{\ex}(q)$ and $\psi_\no(q)$:
\begin{enumerate}
\item The {\it excess contribution} comes from the fact that \emph{every} surface in the family $\mathcal{X}$ is Noether-Lefschetz special; since every $X_t$ is a base change of $R$, they will all have an extra curve class orthogonal to $\ZF(X)$ for every element of $\ZF^\perp(R)\cong\E8(-1)$.
This means that the image $j^{\II}(T)\subset(\D(\Lambda)/\Gamma)^{\II}$ actually lies entirely inside the Noether-Lefschetz cycle $\overline\NL_n$, for $n$ even, and the intersection number $[\overline{\NL}_n]\cdot\alpha$ can be computed using Fulton's excess intersection formula, and then assembled to the series $\psi_{\ex}(q)$.
\item The {\it nodal contribution} comes from the fact that, aside from having an extra section, an elliptic surface can also have jumping Picard rank by having a reducible fiber.
As mentioned in Remark~\ref{rmk:brieskorn}, this will occur when a member of our family $\mathcal{X}$ acquires a node.
The extra algebraic cycles on these nodal surfaces will be counted by $\psi_\no(q)$.
\end{enumerate}

To separate the contributions of $\psi_\ex$ and $\psi_\no$ to $\varphi$, we would like to decompose each $\overline\NL_n$ into separate components, whose intersections with $j^\II(T)$ give $\psi_\ex$, $\psi_\no$, and $\varphi-\psi_\ex-\psi_\no$, respectively.
Unfortunately, we cannot do this directly, since these components can only be distinguished in a small neighborhood of $j^\II(T)$.
A second difficulty is that the space $(\D(\Lambda)/\Gamma)^\II$ is singular, albeit smooth as an orbifold. To make the required topological intersection arguments, we will pass to a smooth finite cover.

\subsubsection{Passing to a finite cover}
Choose a neat congruence subgroup $\Gamma_p\subset\Gamma$.
The singular variety $\D(\Lambda)/\Gamma$ has a finite branched cover from the smooth variety $\D(\Lambda)/\Gamma_p$, which admits a smooth partial toroidal compactification $(\D(\Lambda)/\Gamma_p)^\II$.
By Proposition~\ref{prop:partial-toroidal-functorial}, we have an induced map $\pi\colon(\D(\Lambda)/\Gamma_p)^\II\to(\D(\Lambda)/\Gamma)^\II$.
\begin{prop}
  The action of $\Gamma/\Gamma_p$ on $\D(\Lambda)/\Gamma_p$ extends to $(\D(\Lambda)/\Gamma_p)^\II$, with quotient $\pi\colon(\D(\Lambda)/\Gamma_p)^\II\to(\D(\Lambda)/\Gamma)^\II$.
\end{prop}
\begin{proof}
  This follows from \cite[Lemma~2.6]{harris-functorial}.
\end{proof}

The map $j^\II\colon T\to(\D(\Lambda)/\Gamma)^\II$ does not lift to a map $T\to[\D(\Lambda)/\Gamma]^\II$ to the compactified quotient stack.
This is because in order to define the period map $j^\II$ at the nodal fibers $X_t$ of $\mathcal{X}$, we must resolve the fiber to obtain a smooth K3.
However, after taking the double cover $T^{(1)}\to T$ branched at these basepoints, we can form the Brieskorn simultaneous resolution family $\mathcal{X}^{(1)}\to T^{(1)}$ and hence a lift $j^\II_{(1)} \colon T^{(1)}\to[\D(\Lambda)/\Gamma]^\II$. Now form the base change
\[
  \begin{tikzcd}
    \wt T\pb\ar[r, "\tildej"]\ar[d]&(\D(\Lambda)/\Gamma_p)^\II\ar[d]\\
    T^{(1)}\ar[r, "j^\II_{(1)}"]\ar[d, "2:1"']&{[\D(\Lambda)/\Gamma]}^\II\ar[d]\\
    T\ar[r, "j^\II"]&(\D(\Lambda)/\Gamma)^\II.
  \end{tikzcd}
\]
Since $(\D(\Lambda)/\Gamma_p)^\II\to[\D(\Lambda)/\Gamma]^\II$ is an étale morphism of stacks, $\wt T$ is an étale cover of $T^{(1)}$ of degree $\abs{\Gamma/\Gamma_p}$, and hence $\pi_T\colon\wt T\to T$ is a degree $d \defeq 2\abs{\Gamma/\Gamma_p}$ branched cover.

We then have $d\cdot [T]=(\pi_T)_*[\wt T]$, so $d\cdot j^\II_*[T]=\pi_*(\tildej_*[\wt T])$.
There are Noether-Lefschetz cycles $[\overline{\NL}_{p,n}]$ in $(\D(\Lambda)/\Gamma_p)^\II$, defined in the same way as those in $(\D(\Lambda)/\Gamma)^\II$, such that $[\overline{\NL}_{p,n}]=\pi^*[\overline{\NL}_n]$. It follows that the desired intersection numbers $j^\II_*[T]\cdot[\overline{\NL}_n]$ can be computed as $\frac{1}{d}\cdot\tildej_*[\wt T]\cdot[\overline{\NL}_{p,n}]$.

\subsubsection{Intersecting with a small neighborhood}
Choose a small neighborhood $U$ of $\tildej(\wt T)$ in $(\D(\Lambda)/\Gamma_p)^\II$.
We will now show that $\overline{\NL}_{p,n}\cap U$ is a union of pairwise-transverse complex submanifolds, each of which either contains $\tildej(\wt T)$ or is transverse to it.

Let $\wt{\mathcal{X}}\to \wt{T}$ be the base change of $\mathcal{X}\to T$ to $\wt T$, and let $\wt{\mathcal{X}}_{\reg}\to\wt T_\reg$ be the restriction to the regular locus, i.e., those $t\in\wt T$ with $\wt X_t$ a smooth K3 surface.
As shown in \S\ref{subsec:families-of-k3s}, the complement of $\wt{\mathcal{X}}_{\reg}$ consists both of Type II degenerations (surfaces of the form $R\cup_ER$) and nodal surfaces, but by simultaneous resolution, we can replace the nodal surfaces with smooth ones so that the complement of $\wt{\mathcal{X}}_{\reg}$ only includes the Type II degenerations.

Consider the local systems $\set{\Lambda_R^t}_{t\in\wt T_\reg}\subset\set{\Lambda^t}_{t\in\wt T_\reg}\subset\set{\Hm^2(\wt X_t)}_{t\in\wt T_\reg}$ defined as follows.
Let $\Lambda^t$ be the orthogonal complement of the zero section and fiber class inside $\Hm^2(\wt X_t)$.
Next, recalling that $X_t\to \P^1$ is pulled back from $R\to\P^1$ via a double cover $u_t\colon\P^1\to\P^1$, we let $\Lambda_R^t\subset\Hm^2(X_t)$ consist of the pullbacks of classes in $\ZF^\perp(R)$, so that $\Lambda_R^t\cong\E8(-2)$.

\begin{prop}
    The local system $\set{\Lambda_R^t}_{t\in\wt T_\reg}$ is trivial.
\end{prop}

\begin{proof}
It suffices to show that the corresponding local system on $T_{\reg}$ has trivial monodromy, since $\set{\Lambda_R^t}_{t\in\wt T_\reg}$ is pulled back from it.
Recall that $T\simeq \P^1$ is simply connected, and the local monodromies are trivial by the easy direction of the Invariant Cycle Theorem, since $\mathcal X$ was constructed as a double cover $\mathcal X \to R\times \P^1$.
\end{proof}

For each connected component of $\wt T_\reg$, fix a basepoint $t_0\in\wt T_\reg$ and fix an identification $\Lambda_{\K3}$ with $\Hm^2(\wt X_{t_0})$, and hence of $\Lambda$ with $\Lambda^{t_0}$.
We write $\Lambda_R$ for $\Lambda_R^{t_0}\subset\Lambda$.

The identification of $\Lambda$ with $\Lambda^{t_0}$ also determines a lift $\tau_0\in\D(\Lambda)$ of $\tildej(t_0)\in\D(\Lambda)/\Gamma_p$.
Now consider the universal cover $\wt T_\reg^{\mathrm{univ}}$ and the corresponding lift $\wt T_\reg^{\mathrm{univ}}\to\D(\Lambda)$ of $\rstr{\tildej}{\wt T_\reg}$, whose image lies entirely inside $v^\perp\subset\D(\Lambda)$ for each $v\in\Lambda_R$.

The neighborhood $U\cap\D(\Lambda)/\Gamma_p$ of $\tildej(\wt T_\reg)$ lifts to a $G\defeq\tildej_*(\pi_1(\wt T_\reg))$-invariant neighborhood $\wt U$ of the image of $\wt T_\reg^{\mathrm{univ}}\to\D(\Lambda)$.
Since the local system $\set{\Lambda_R^t}_{t\in\wt T_\reg}$ is trivial, it follows that for each $v\in\Lambda_R$, the intersection $M_v\defeq v^\perp\cap\wt U$ is invariant under the action of $G\subset\Gamma_p$.

Thus, for each $v$, we have an embedding $M_v/G\hto U\cap\D(\Lambda)/\Gamma_p$.
Next, we consider the closure $\ol{M_v/G}$ of $M_v/G$ in $U$.

\begin{prop}
$\ol{M_v/G}\subset(\D(\Lambda)/\Gamma_p)^\II$ is a smooth submanifold.
\end{prop}

\begin{proof}
We use the description in \cite{looijenga-arrangements} of a neighborhood of the toroidal boundary. Let $J\subset \Lambda$ be an isotropic plane, and consider the linear projections $\P(\Lambda_\C) \dasharrow \P(\Lambda_\C/J_\C) \dasharrow \P(\Lambda_\C/J^\perp_\C)$. Since $\D$ is an open subset of the isotropic quadric in $\P(\Lambda_\C)$, these projections restrict to regular fibrations on $\D$:
$$\D \to \pi_J(\D) \to \pi_{J^\perp}(\D)\simeq \mathbb H.$$
The fibers of the first map are upper half-planes $\mathbb H$, and the fibers of the second map are $(J^\perp/J)_\C\simeq J(\R)\otimes J^\perp/J$. A neighborhood of the Type II cusp in $\D/\Gamma_p$ is isomorphic to $\D/\Gamma_{p,J}$, where $\Gamma_{p,J}\subset\Gamma_p$ is the stabilizer of $J$. This is a maximal parabolic over $\Z$ with Levi decomposition
$$1\to N_{p,J} \to \Gamma_{p,J} \to \Gamma(p,J) \to 1,$$
where $\Gamma(p,J)\subset GL(J)$ is a congruence subgroup, and $N_{p,J}$ is an arithmetic Heisenberg group:
$$0 \to \wedge^2 J\simeq \Z \to N_{p,J} \to J\otimes J^\perp/J \to 0.$$
This filtration of $\Gamma_{p,J}$ is compatible with the tower above:
$$\D/\Gamma_{p,J} \to \pi_J(\D)/(\Gamma_{p,J}/\Z) \to \pi_{J^\perp}(\D)/\Gamma(p,J).$$
The fibers of the first map are punctured disks $\Delta^*\simeq \mathbb H/\Z$, and the fibers of the second map are tori $J(\R)/J(\Z)\otimes J^\perp/J$. The partial toroidal compactification is obtained by filling in these punctured disk fibers. Note that locally, we have $\Delta^*\times \mathbb T^{32}\subset \Delta\times \mathbb T^{32}$, so in particular the compactification is smooth.

To analyze the Noether-Lefschetz loci, if $v\notin J^\perp$, then $v^\perp\subset \D$ surjects onto $\pi_J(\D)$ via the linear projection, so the Noether-Lefschetz divisor does not meet the Type II boundary. If $v\in J^\perp$, then $\pi_J(v^\perp)\subset \pi_J(\D)$ is a hypersurface which restricts to a hyperplane in each fiber $J(\R)\otimes J^\perp/J$ defined by $[v]^\perp$ inside $J^\perp/J$. Since the vector is integral, the dual hyperlattice descends to a hypertorus $\mathbb T^{30}\subset \mathbb T^{32}$. The closure of this hypertorus in the compactification is given by
\[\Delta\times [v]^\perp \subset \Delta\times\mathbb T^{32}. \qedhere\]
\end{proof}

\subsubsection{Computing the intersection numbers}\label{subsubsec:computing-numbers}
\begin{prop}\label{prop:local-normal-bundle}
  The normal bundle to $\ol{M_v/G}$ is the restriction of the dual Hodge bundle $\lambda^\vee$ of $(\D(\Lambda)/\Gamma_p)^\II$, and we have $\tildej^\II_*[\wt T]\cdot[\ol{M_v/G}]=-d$, where $[\ol{M_v/G}]$ is a class in the Borel-Moore homology of $U$.
\end{prop}

\begin{proof}
The fact that $M_v/G$ has normal bundle $\lambda^\vee$ follows from the definition of $M_v = v^\perp \cap U$ as a hyperplane section, and applying descent. The Hodge bundle $\lambda$ on $\D(\Lambda)/\Gamma_p$ extends to the partial toroidal compactification $(\D(\Lambda)/\Gamma_p)^\II$ where it is trivial on the torus fibers $\mathbb T^{32}$, since it is defined as the pullback of a line bundle from the Baily-Borel compactification. The normal bundle of $\overline{M_v/G}$ is also trivial when restricted to the boundary, since it meets the boundary in a hypertorus. The result now follows from the localization sequence for Picard groups:
$$\Z[U\cap \partial( \D(\Lambda)/\Gamma_p)] \to \Pic(U) \to \Pic(U\cap \D(\Lambda)/\Gamma_p)) \to 0$$
where $\partial (\D(\Lambda)/\Gamma_p)=(\D(\Lambda)/\Gamma_p)^\II - \D(\Lambda)/\Gamma_p$. Indeed, the difference between $\lambda^\vee$ and the normal bundle to $\overline{M_v/G}$ is a multiple of $[U\cap \partial( \D(\Lambda)/\Gamma_p)]$, but if we restrict the difference to $U\cap \partial( \D(\Lambda)/\Gamma_p)$ itself, the result is trivial. On the other hand, the normal bundle of $U\cap \partial( \D(\Lambda)/\Gamma_p)$ is anti-ample on each $\mathbb T^{32}$ by its construction as an exceptional divisor. The numerical statement follows from the excess intersection formula since the Hodge bundle of $\mathcal X\to T\cong\P^1$ is $\O(1)$ (by Proposition~\ref{prop:hodgedegree}), so the Hodge bundle of $\wt{\mathcal X}\to \wt T$ has degree $d$.
\end{proof}

We are now ready to separate the components of $\ol{\NL}_{p,n}$ that contain $\tildej(\wt T)$ and contribute to $\psi_\ex$ from those that intersect it transversely and contribute to $\psi_\no$ and $\phi-\psi$.

\begin{prop}\label{prop:component-decomp}
$\ol{\NL}_{p,n}\cap U$ is a union of pairwise-transverse connected complex submanifolds
\[
\Big(
    \bigcup_{t\in\wt T}\bigcup_{\alpha\in I_t}N_\alpha
\Big)
\cup
\bigcup_{\substack{v\in \Lambda_R\\v^2=-2n}}
\ol{M_v/G}
\]
where each $N_\alpha$ with $\alpha\in I_t$ intersects $\tildej^\II(\wt T)$ transversely, and where the index set $I_t$ is the set of divisor classes in $\ZF^\perp(X_t)$ with self-intersection $-2n$ and lying outside of $\Lambda_R^t$.
\end{prop}

\begin{proof}
  Fix $t\in\wt T_\reg$ and consider a lift $\tau\in\D(\Lambda)$ of $\tildej(t)$.
  The components of $\ol{\NL}_{p,n}\cap U$ containing $\tildej(t)$ are precisely the images of components $v^\perp\cap\wt U$ containing $\tau$ for $v\in\ZF^\perp(X_t)$ with $v^2=-2n$.
Those components $v^\perp\cap\wt U$ with $v\in\Lambda_R^t$ correspond to the $\ol{M_v/G}$, so the remaining components $N_\alpha$ are in bijection with the remaining $v\in\Lambda$ with $v^2=-2n$, as claimed.

  It remains to prove the transversality claim.
  Near any $t\in\wt T$ with $\tildej(t)\in\D(\Lambda)/\Gamma_p$ (this is the case when $\tildej(t)$ lies in some $N_\alpha$), we can choose a small neighborhood $\wt W$ of $t$ and a lift $\hatj\colon\wt W\to\D(\Lambda)$ of $\tildej$.
  Transversality of $\tildej$ to $N_\alpha$ is equivalent to transversality of $\hatj$ to the corresponding component $v^\perp\cap\wt U$.
  We now consider two cases, according to whether or not $\pi_T$ is ramified at $t$.

  The first is when $t\in\wt T$ is not a ramification point of $\wt T\to T$.
  Suppose for the sake of contradiction that $\hatj$ fails to be transverse to some $v^\perp\cap\wt U$ at $t$.
We then have a non-trivial tangent vector $\trait=\Spec(\C[\varepsilon]/\varepsilon^2)\to W$, where $W\subset T$ is the image of $\wt W$ in $T$, whose composition $\trait\to\D(\Lambda)$ with the map $W\to\D(\Lambda)$ lies in $v^\perp$.
Pulling back $\mathcal{X}$, this gives a K3 surface $X_{\trait}$ over $\trait$ containing an extra divisor class.

  From the Shioda-Tate exact sequence (\ref{subsec:elliptic-surfaces}), this means that $X_{\trait}$ either has a reducible fiber or an extra section.
  We must be in the latter case, as the former only occurs when $t$ is a ramification point of $\wt T\to T$.
  It follows that $X_{\trait}$ has an extra section, which maps to a bisection under the double cover $X_{\trait}\to R\times\trait$.
This implies that the map $T\to\Sym^2(\P^1)$ fails to be transverse to the Severi curve $f_L\colon\wt V(L)\to\Sym^2(\P^1)$, which is impossible by Theorem~\ref{thm:equisingular} as $T$ was chosen to be a general conic.

  The second case is when $t$ is a ramification point of degree 2 of $\wt T\to T$, so that $\wt X_t$ has an extra ($-2$)-curve class in a reducible fiber.
  By construction, there is an $\mathrm A_1$-singularity in the surface $X_{\pi_T(t)}$, $\pi_T^*\mathcal{X}\to\wt T$ has a threefold node singularity lying over $t$, and $\wt{\mathcal X}$ is a small resolution of $\pi_T^*\mathcal{X}$. The exceptional curve in $\wt{\mathcal X}$ has normal bundle isomorphic to $\O_{\P^1}(-1)^{\oplus 2}$, so in particular the curve is infinitesimally rigid. It follows that the infinitesimal period map $\trait\to\wt T\to\D(\Lambda)$ will be transverse to the relevant Noether-Lefschetz divisor, as desired.
\end{proof}

It follows from Proposition~\ref{prop:component-decomp} that the intersection number $\tildej_*[\wt T]\cdot\ol{\NL}_{p,n}=d\cdot j_*[T]\cdot\ol{\NL}_{n}$ decomposes as a sum $\chi_{p,n}+\psi_{\ex,p,n}$, defined as follows:
\[
  \tildej_*[\wt T]\cdot\ol{\NL}_{p,n}=
  \sum_{t\in\wt T}\sum_{\alpha\in I_t}\tildej_*[\wt T]\cdot[N_\alpha]
  +
  \sum_{\substack{v\in \Lambda_R\\v^2=-2n}}
  \tildej_*[\wt T]\cdot[\ol{M_v/G}]
  \qefed
  \chi_{p,n}+\psi_{\ex,p,n}
\]
Thus, defining the series $\psi_\ex(q)\defeq\frac1d\sum_{n=0}^\infty\psi_{\ex,p,n}q^n$ and $\chi(q)=\frac1d\sum_{n=0}^\infty\chi_{p,n}q^n$, we have
\begin{equation}\label{eq:psi-chi-eq}
  \sum_{n=0}^\infty j_*[T]\cdot[\ol{\NL}_n]q^n=\psi_\ex(q)+\chi(q)
\end{equation}
By Proposition~\ref{prop:local-normal-bundle}, we have that $\psi_{\ex,p,n}$ is $-d$ times the number of nonzero vectors of length $-2n$ in $\Lambda_R\cong\E8(-2)$, from which it follows by Proposition~\ref{prop:e8-and-e4} that
\begin{equation}
  \label{eq:excess-contr}
  \psi_\ex(q)=-\eis_4(q^2)+1.
\end{equation}

Next, $\chi(q)$ further decomposes as a sum
\begin{equation}\label{eq:chi-decomp}
  \chi(q)=\psi_{\sec}(q)+\psi_{\no}(q)=
  \frac1d\sum_{n=0}^\infty\psi_{\sec,p,n}q^n+
  \frac1d\sum_{n=0}^\infty\psi_{\no,p,n}q^n,
\end{equation}
defined as follows.
By Proposition~\ref{prop:component-decomp} $\chi_{p,n}$ is the number of divisor classes of self-intersection number $-2n$ lying on surfaces $\wt X_t$ in the family $\wt{\mathcal{X}}$ orthogonal to $\ZF(X_t)$ and lying outside of $\Lambda_R^t$.
We let $\psi_{\sec,p,n}$ be the count of such curves which are sections, and we let $\psi_{\no,p,n}$ be the count of the remaining curves.

As explained at the beginning of \S\ref{sec:bounding}, we have that
\[
  \psi_{\sec}(q)=4\sum_L\deg\left({{f_L}_*[\wt V(L)]}\right)q^{n+2}
\]
since, by Proposition~\ref{prop:galois}, sections in the family $\mathcal{X}$ (and hence $\wt{\mathcal{X}}$) that are not pulled back from $R$ occur in pairs and precisely when the double cover $u_t\colon\P^1\to\P^1$ defining some $X_t$ has the same branch locus as a bisection on $R$; and these correspond to the intersection points of ${f_L}_*(\wt V(L))$ with the conic $T\to\Sym^2(\P^1)$ which tracks the branch points of the $u_t$.

Finally, recalling the Jacobi theta function \( \theta \) from \S\ref{subsec:qmod-forms}, we have
\begin{prop}\label{prop:nodal-contribution}
  \begin{equation}
    \label{eq:nodal-contr}
    \psi_\no(q)=
    12\big(\theta(q)-1\big)\cdot\eis_4(q^2) =
    \eis_4(q^2)\cdot 24\sum_{k\geq 1}q^{k^2}.
  \end{equation}
\end{prop}
\begin{proof}
  Any curve counted by $\psi_{\no,p,n}$ must occur on a surface $\wt X_t$ in $\wt{\mathcal{X}}$ which is a resolution of a nodal surface in $\pi_T^*\mathcal{X}$.
  There are 12 lines in $\P^2$ corresponding to the 12 critical values of $\pi\colon R \to \P^1$.
  The test conic $T \to \P^2$ intersects each of them twice.
  Hence, since $\pi_T$ is branched at each such $t\in\wt T$, there are $24\cdot\frac d2$ such surfaces $\wt X_t$.

  Now fix one such $\wt X_t$, and denote by $e$ the exceptional $-2$ curve class of the resolution $\wt X_t\to(\pi^* X)_t$.
  For each vector $v\in \E8(-2)$ there is a corresponding section class $s_v\in \NS(\wt{X}_t)$.
  For each \( k \ne 0 \), the algebraic cycle $s_v+ke$ has orthogonal projection $v+ke\in\ZF^\perp(\wt X_t)$, which lies outside $\Lambda_R^t$ and has self-intersection $v^2-2k^2$.

  We thus receive \( (24 \frac d 2) \cdot 2 \) contributions to \( \psi_{\no,p,n} \) for each vector \( v \) in \( \E8 \) and each \( k \ge 1 \) with \( -2 v^2 - 2 k^2 = - 2 n \).
  Using Proposition~\ref{prop:e8-and-e4}, the claim follows.
\end{proof}

\begin{proof}[Proof of Proposition~\ref{prop:degrees-are-modular}]
  The claim now follows from Equations~\eqref{eq:psi-chi-eq},~\eqref{eq:excess-contr}~\eqref{eq:chi-decomp},~and~\eqref{eq:nodal-contr}, since $\eis_4(q^2)$ and $12(\theta(q)-1)\cdot\eis_4(q^2)$ are modular forms of weight $\le 5$ for $\Gamma_0(4)$.
\end{proof}

\subsection{Proof of the genus bound}\label{subsec:mainthm}
\begin{lem}\label{lem:f_l_is_birational}
  For fixed \( g \) and \( \star \) as in Corollary~\ref{cor:mres-is-smooth}, and for a general point \( (R,L) \in \MRES_g^\star \), the branch morphism $f_L\colon\wt{V}(L)\to\Sym^2(\P^1)$ is birational onto its image.
\end{lem}
\begin{proof}
  Let $R$ be a general rational elliptic surface and $L$ a bisection line bundle on $R$. Consider two distinct points $B_1,B_2$ in the smooth locus of $V(L)\subset |L|$ with the same image under $f_L$. That is, $B_1$ and $B_2$ are rational bisections of $R \to \P^1$ with the same branch divisor $p_1+p_2\in\Sym^2(\P^1)$.
  Consider a double cover $u\colon\P^1\to\P^1$ branched at $p_1+p_2$, and the resulting K3 surface $X$ as discussed in \S\ref{subsec:families-of-k3s}.

  We claim that the Picard rank of such an $X$ is $\ge 12$.
  For $i=1,2$, the bisection $B_i$ pulls back to a pair of Galois conjugate sections $(u')\I(B_i)=s_i+s_i'$ in $X$.
  Note that for any section $s$ of $R$, we have $B_i\cdot s= s_i\cdot(u')\I(s)=s_i'\cdot (u')\I(s)$, and $B_1\cdot s=B_2\cdot s$ as $B_1$ and $B_2$ are rationally equivalent.
  In other words, $s_1,s_2,s_1',s_2'$ all have the same intersection number with $(u')\I(s)$ for any given $s\in \MW(R)$.

  Recall from \S\ref{subsec:elliptic-surfaces} that $\ZF^\perp(X)$ is a negative-definite lattice, and we have an isomorphism $\MW(X)\to\ZF^\perp(X)$ of abelian groups.
  The images in $\ZF^\perp(X)$ of pullbacks $(u')\I(s)$ form a sublattice $\Lambda_R$ isomorphic to $\E8(-2)$.
  Since $s_1,s_2,s_1',s_2'$ all have the same height, their orthogonal projections to $\ZF^\perp(X)$ all have the same norm, as in Proposition~\ref{prop:sec-proj-formula}, and furthermore their orthogonal projections to $\E8(-2)$ are all equal. It follows that $\ZF^\perp(X)$ must have rank $\ge 10$, since if it only had rank \( 9 \), the four vectors in question would all have to lie on the same line, and intersect the same sphere, but a line intersects a sphere in at most two points.

  Recall from Corllary~\ref{cor:severi-invariance} that we have a universal family \( \mathcal{V}_g^\star \to \MRES_g^\star \) over the 8-dimensional moduli stack \( \MRES_g^{\star} \).
  The various maps $f_L$ extend to a universal map $F\colon \mathcal{V}_g^\star\dasharrow\P^2$.
  Since the property of $\rstr{F}{V(R,L)}$ being birational is open, it suffices to show that $\MRES_g^\star$ contains at least one $(R,L)$ for which \( f_L \) is birational.

  Now suppose for the sake of contradiction that $f_L$ fails to be birational onto its image for \emph{all} $(R,L)\in \MRES_g^\star$.
  Then for each $(R,L)$, there exists a curve in
  \[\wt{V}(R,L)\times_{\P^2}\wt{V}(R,L)\smallsetminus\Delta.\]
  In other words, there is a 1-parameter family of $B_1\in V(R,L)$ such that there exists a distinct $B_2\in V(R,L)$ with $f_L(B_1)=f_L(B_2)$.
  There exists an open subset $W\subset\Sym^2(\P^1)$ over which the $\mu_2$-gerbe $\Mb_0(\P^1,2)\to\Sym^2(\P^1)$ has a section.
  By passing to an open subset $\mathcal{W}\subset\mathcal{V}_g^\star$, we may assume that $F(\mathcal W)\subset W$, and hence that we have a lift $\wt{F}\colon\mathcal{W}\to\Mb_0(\P^1,2)$.
  Base changing each \( R \) appearing in \( \mathcal{W} \) along the double-cover of \( \P^1 \) specified by \( \wt{F} \) gives us a family of K3 surfaces over $\mathcal{W}$, each with Picard rank $\ge12$.
  Let $\mu\colon\mathcal{W}\to\mathcal{M}^M_{\K3}$ be the corresponding map to the moduli space of $M$-polarized K3 surfaces, where $M$ is the appropriate polarizing lattice of signature $(1,11)$.
  We will show that $\mu$ has discrete fibers, and hence obtain a contradiction, since $\mathcal{M}^M_{\K3}$ is 8-dimensional, while $\mathcal{W}$ is 9-dimensional.

  It is enough to show that $\mu\I([X])$ is countable for any $M$-polarized K3 surface $X$.
  Given any $(R,L,B)\in\mu\I([X])$, we know that $R$ is the quotient of $X$ by some involution.
  Since each K3 surface $X$ has discrete automorphism group, it follows that there are only countably many such $R$.
  Next, for fixed $R$, there are only countably many bisection line bundles $L$ on $R$ at all (hence only countably many $L$ such that $(R,L,B)\in \mu\I([X])$ for some $B$).

  Now fix $R,L$, and suppose that there is a positive-dimensional component $Z$ of $\mu\I([X])$ consisting of points of the form $(R,L,B)$.
  The restriction $\rstr{\mu}Z\colon Z\to\mathcal{M}^{M}_{\K3}$ gives a family of K3 surfaces over $Z$ all of whose fibers are isomorphic to $X$ and are pullbacks of $R$ along various degree 2 covers of $\P^1$.
  After passing to an étale open of $Z$, we may assume that this family is trivial, since $\mathcal{M}^M_{\K3}$ is a DM-stack.

  But each fiber $X$ of the trivial family comes with a map $u'\colon X\to R$ branched along the fibers of the branch divisor of $B$, hence $B$ with different branch divisors give rise to different $u'$, and hence different involutions of $X$, of which there are only countably many.
  Since $X$ is not uniruled, there can only be finitely many $B$ with a given branch divisor.
  It follows that that there are only countably many $B$, as desired.
\end{proof}

\begin{thm}\label{thm:genusbound-with-proof}
For any smooth rational elliptic surface \( R \) with irreducible fibers and bisection line bundle \( L \) on \( R \), the geometric genus of a Severi curve $V(R,L)$ is $O(g^{12+\varepsilon})$, where $g = \frac{1}{2}L^2$ is the genus of a general member $B\in |L|$.
\end{thm}
\begin{proof}
  Fix a bisection line bundle $L$ of genus $g$, and let us first consider the case in which $g$ is odd. By Lemma~\ref{lem:min-height-realizable} below, we can choose a zero section $Z$ of $R$ so that $L$ has height $g$. Now, by Proposition~\ref{prop:degrees-are-modular} and the fact that $\varphi(q)$ is quasi-modular of weight 10, we have an $O(g^{9+\varepsilon})$ bound on the sum $\sum_{L'}\deg({f_{L'}}_*[\wt V(L')])$ over \emph{all} height $g$ bisection line bundles $L'$. Hence, in particular, we have a $O(g^{9+\varepsilon})$ bound on $\deg({f_L}_*[\wt V(L)])$.

  We can improve this crude estimate as follows. By Corollary~\ref{cor:severi-invariance}, we have that $\deg({f_L}_*[\wt V(L)])$ in fact only depends on the genus $g$ of $L$, since $g$ is odd. Hence, all terms $\deg({f_{L'}}_*[\wt V(L')])$ indexed by $L'$ of genus $g$ -- of which there are $\#\mathcal{B}_{g,g}$ -- are equal. This gives
  \[\label{eq:genusbound-inequality}
  \deg({f_L}_*[\wt V(L)]) = \frac{1}{\#\mathcal{B}_{g,g}}\sum_{L'}\deg({f_{L'}}_*[\wt V(L')]).
  \]
  By Lemma~\ref{lem:g-g-bisections-biject}, $\#\mathcal{B}_{g,g}$ is the $q^{g+2}$ coefficient $240\sigma_3(g+2)$ of the Eisenstein series $\eis_4(q)$, which is bounded below by $g^3$. We hence obtain an $O(g^{6+\varepsilon})$ bound on $\deg({f_L}_*[\wt V(L)])$.
  By the quadratic genus-degree formula, this gives us a $O(g^{12+\varepsilon})$ bound on the geometric genus of the image of $f_L$.
  By Lemma~\ref{lem:f_l_is_birational}, the latter is birational to \( V(L) \) for \emph{general} \( (R,L) \) of genus \( g \), hence we obtain the same bound on the genus of \( V(L) \) for general \( (R,L) \) of genus \( g \).
  But since geometric genus is lower semi-continuous over a reduced base, and \( \MRES_g^\ord \) is reduced by Corollary~\ref{cor:mres-is-smooth}, we get the same bound for \emph{all} \( (R,L) \) of genus \( g \).

 The proof for $g$ even is similar, except for the additional distinction of ordinary versus Weierstrass line bundles $L$. When $g\equiv 2 \mod 4$, the inequalities analogous to \eqref{eq:genusbound-inequality} become
\begin{align*}
\deg({f_L}_*[\wt V(L)]) &\le \frac{1}{\#\mathcal{B}^\ord_{g,g}}\sum_{L'}\deg({f_{L'}}_*[\wt V(L')]);\\
\deg({f_L}_*[\wt V(L)]) &\le \frac{1}{\#\mathcal{B}^\Wei_{g,g}}\sum_{L'}\deg({f_{L'}}_*[\wt V(L')]).
\end{align*}
When $g\equiv 0\mod 4$, we use Lemma~\ref{lem:g-g-1-bisections-biject} for the Weierstrass case:
\[
\deg({f_L}_*[\wt V(L)]) \le \frac{1}{\#\mathcal{B}^\Wei_{g,g-1}}\sum_{L'}\deg({f_{L'}}_*[\wt V(L')]).
\]
In all cases, the denominator is bounded below by $Cg^3$, so we obtain the upper bound of $O(g^{12+\epsilon})$ on the geometric genus of $V(L)$.
\end{proof}

By Proposition~\ref{prop:bisec-proj-formula}, any bisection of genus $g$ and height $n$ satisfies $g\le 2n+2$.
The following lemma shows that, by varying the choice of zero section, the height can be made equal to $g$ or $g-1$. The proof uses some technical facts about the lattice $\E8(-1)$.

\begin{lem}\label{lem:min-height-realizable}
  For any bisection line bundle $L=\O(B)$ of genus $g$, there exists a section $s$ such that $s\cdot B=g-\delta_g$, where $\delta_g\in\set{0,1}$ only depends on $g$. In fact, $\delta_g=1$ if and only if $g \equiv 0 \mod 4$, and $L$ is Weierstrass.
\end{lem}
\begin{proof}
  Using Propositions~\ref{prop:sec-proj-formula}~and~\ref{prop:bisec-proj-formula}, one computes that for any section $s$,
  \begin{equation}\label{eq:latticecomp}
    s\cdot B
    =-\frac14(\Pi(B)-2\Pi(s))^2+\frac{g}2-1.
  \end{equation}
  We now distinguish two cases, according to whether $\Pi(B)\in \ZF^\perp(R)\simeq \E8(-1)$ is 2-divisible.

  If $\Pi(B)=2v$, then let $s$ be the unique section with $\Pi(s)=v+w$, where $w\in\ZF^\perp(R)$ is to be determined.
  Then equation \eqref{eq:latticecomp} becomes
  \[
  s\cdot B=-w^2+\frac{g}{2}-1.
  \]
  It follows that $\frac{g}{2}$ is an integer.
  If $\frac{g}{2}$ is odd, resp. even, we can choose $w$ so that $-w^2=\frac{g}{2}+1$ (hence $\delta_g=0$), resp. $-w^2=\frac{g}{2}$ (hence $\delta_g=1$).

  Now suppose that $\Pi(B)$ is not 2-divisible.
  By Lemma~\ref{lem:parities-classf}, we can choose $u\in\ZF^\perp(R)$ with $u^2\in\set{-2,-4}$ such that $\Pi(B)-u=2v$ for some $v\in\ZF^\perp(R)$.
  Let $s$ be the unique section with $\Pi(s)=v-w$, where $w\in \ZF^\perp(R)$ is to be determined. Then equation \eqref{eq:latticecomp} becomes
  \[
  s\cdot B
  =-\frac{1}{4}(u+2w)^2+\frac{g}{2}-1
  =-\frac{1}{4}u^2-w^2-u\cdot w+\frac{g}{2}-1.
  \]
  We now consider the cases of $g$ being odd or even.
  Note that since $s\cdot B$ is an integer, these correspond precisely to $u^2=-2$ and $u^2=-4$.
  In either case, we will find $w$ such that $s\cdot B=g$.
  Using Lemma~\ref{lem:trans-on-roots}, we can assume that
  \[
    u\in\set{(1^2,0^6),(1^4,0^4)}
  \]
  (superscripts denote repeated entries).
  Now taking $w$ of the form $(0^4,a,b,c,d)$ or $(1,0^3,a,b,c,d)$, and using the fact that every positive integer is a sum of four squares, it is easy to see that we can arrange for $w^2$ to be as required.
\end{proof}

\section{Multiplicity Conjecture}\label{sec:multiplicity-conjecture}

The proof of of the genus bound in Theorem \ref{thm:genusbound-with-proof} uses only the total degree of the images $f_L: \wt{V}(L) \to \Sym^2(\P^1)\simeq \P^2$ for various bisection line bundles $L$. In this section, we give a conjectural formula for the local intersection multiplicities of the image of each $f_L$ with the diagonal conic $\Delta\subset \Sym^2(\P^1)$. The formula is in terms of the limiting curve cycle on $R$ parametrized by the corresponding point in the Severi curve $V(L)$. As evidence, we will show that the conjectural formula implies (and refines) the formula for the quasi-modular generating function $\varphi(q)$ of Section \ref{subsec:families-of-k3s}. It also gives an improved bound on the geometric genus of $V(L)$.

\subsection{Telltales}
Recall that the map $f_L: \wt{V}(L) \to \Sym^2(\P^1)$ tracks the two branch points of the varying rational bisection curve $B\subset R \to \P^1$. The image of $f_L$ intersects the diagonal conic $\Delta\subset \Sym^2(\P^1)$ when branch points collide. This results in a limit cycle $B_0\in |L|$ with worse singularities than the general bisection $B$. We will refer to these special cycles as telltales, and they fall into two types, simple and non-simple.
\begin{defn}
    A simple telltale is a cycle of the form $s_1+s_2$, where the $s_i$ are distinct section curves of $R\to \P^1$.
\end{defn}
\begin{remark}
The reason that $s_1\neq s_2$ in the definition of simple telltale is the fact that $L = \O(2s)$ has a unique section, so the Severi curve is empty.
\end{remark}
\begin{defn}
    A non-simple telltale is a cycle of the form $s_1+s_2+mN$, where $m$ is a positive integer, and $N$ is a nodal fiber of $\pi:R \to \P^1$, and the $s_i$ are possibly equal sections.
\end{defn}
\begin{remark}
Non-simple telltales with $s_1=s_2$ occur only when $L$ is Weierstrass. Note that in that case, telltales only appear for $m\geq 2$ because for $m\leq 1$ the Severi curve for $L=\O(2s+mN)$ is empty.
\end{remark}
Since the generic element $B\in \wt{V}(L)$ is irreducible rational, the limit cycle $B_0$ cannot have any components of positive genus. This is why the fiber component $N$ of the non-simple telltale is required to be rational (and hence nodal, since $R$ is general).
\begin{prop}\label{prop:simple-multiplicity}
Let $B_0\in V(L)$ be a simple telltale cycle $s_1+s_2$, with $s_1\cdot s_2 = g+1$, and suppose that the intersection is transverse. Then $V(L)$ has $g+1$ smooth branches \'{e}tale locally around the point corresponding to $B_0$.
\end{prop}
\begin{proof}
Each local component of $V(L)$ corresponds to choosing a node in $s_1+s_2$ to be smoothed. By Lemma \ref{lem:vanishingtheorem}, the first order deformation space of the rational bisection $s_1+s_2$ preserving the $g$ remaining nodes is 1-dimensional.
\end{proof}
We expect the following conjecture to be true, and while probably not essential, it simplifies the exposition in this section considerably, beginning with Proposition \ref{prop:simple-multiplicity}. A special case has been proved in \cite{ulmer-urzua}.
\begin{conj}
    If $R \to \P^1$ is a general rational elliptic surface, then any two sections meet transversely.
\end{conj}

Because the sections $s_i$ are infinitesimally rigid on $R$, the $g+1$ points of $\wt{V}(L)$ lying over a simple telltale meet the diagonal $\Delta\subset \Sym^2(\P^1)$ transversely via the map $f_L$. Hence, each simple telltale will contribute $g+1$ to the total degree
$$f_{L*}[\wt{V}(L)]\cdot [\Delta]\in \mathbb N$$
We now formulate our conjecture for the intersection multiplicity from non-simple telltales.
\begin{conj}\label{conj:multconj}
The intersection multiplicity of $f_L$ with the diagonal conic $\Delta$ at a non-simple telltale $s_1+s_2+mN$ is given by
\begin{align*}
    &2\sigma_1(m) &s_1\neq s_2; \\
    &\sigma_1(m) &s_1=s_2,\, m\neq \mathrm{square};\\
    &\sigma_1(m)-1 &s_1=s_2,\, m= \mathrm{square}.
\end{align*}
\end{conj}

\subsection{An improved genus bound}
Assuming Conjecture~\ref{conj:multconj}, we obtain an improved upper bound for the geometric genus of $V(L)$.
It allows us to compute the individual degrees $\deg({f_L}_*[\wt V(L)])$; this is in contrast to the proof of Theorem~\ref{thm:genusbound-with-proof}, in which we are only able to compute the sum of these degrees over all bisection line bundles of a fixed height.

\begin{prop}\label{prop:conjectural-degree-formula}
  If Conjecture~\ref{conj:multconj} holds, then for a general pair $(R,L)$, $2\deg({f_L}_*[\wt V(L)])$ is the coefficient of $q^{g(L)+2}$ in the series
  \begin{equation}\label{eq:conjectural-degree-formula}
    D\eis_{4,\infty} - \eis_{4,\infty} \eis_2(q^2)
    \quad\text{ or }\quad
    \frac12\Big(
    D\eis_4(q^4)
    -\eis_4(q^4)\eis_2(q^2)
    -12\big(\theta(q^2)-1\big)
    \Big)
  \end{equation}
  % When you work it out, it is natural to stick an extra $-1$ at the end inside the second parentheses, but it doesn't affect the proposition, and it's nice later not to have it.
  according to whether $L$ is ordinary or Weierstrass.
\end{prop}

Using the quadratic genus-degree formula as in the proof of Theorem~\ref{thm:genusbound}, we then obtain:
\begin{cor}
  If Conjecture~\ref{conj:multconj} holds, then the geometric genus of a Severi curve $V(L)$ is $O(g^{10+\epsilon})$, where $g = \frac{1}{2}L^2$.
\end{cor}

\begin{proof}[Proof of Proposition~\ref{prop:conjectural-degree-formula}]
  Conjecture~\ref{conj:multconj} gives the intersection multiplicities of $f_L$ with $\Delta$ for a telltale of any given fiber multiplicity.
  Hence, to compute $f_L\cdot[\Delta]$, it remains to determine the number of telltales of each fiber multiplicity, which we do below in Proposition~\ref{prop:telltale-counting-series}.

  Writing $E_{4,\infty}=\sum_{k=0}^\infty a_kq^k$ and recalling that $\eis_2(q)=1-24\sum_{k=1}^\infty\sigma_1(k)q^k$, equation \eqref{eq:conjectural-degree-formula} says that, for $L$ ordinary of genus $g$, the intersection number $f_L\cdot[\Delta]$ should be given by
  \[
    \begin{split}
      &(g+2)a_{g+2}-\Big(a_{g+2}-24\sum_{m=1}^{\floor{\frac{g+2}2}}a_{g+2-2m}\sigma_1(m)\Big)\\
      =&(g+1)a_{g+2}+\sum_{m=1}^{\floor{\frac{g+2}2}}(12a_{g+2-2m})\cdot\big(2\sigma_1(m)\big).
    \end{split}
  \]
  The first of these terms is the contribution to $f_L\cdot[\Delta]$ coming from the simple telltales according to Propositions~\ref{prop:simple-multiplicity}~and~\ref{prop:telltale-counting-series}, and the second term is the contribution coming from the non-simple telltales according to Conjecture~\ref{conj:multconj} and Proposition~\ref{prop:telltale-counting-series}.

  The case of when $L$ is Weierstrass is similar, where we use the second power series formula from Proposition~\ref{prop:telltale-counting-series}, except that we need to make a correction according to Conjecture~\ref{conj:multconj} for telltales of the form $B=2s+mN$.
  A given Weierstrass line bundle will contain exactly 12 such telltales in its complete linear system, since $\Pi(B)$ is divisible by 2 and $s$ is the corresponding section, and $m=(g+2)/2$ by the adjunction formula.

  Thus, instead of contributing $2\sigma_1(\frac{g+2}{2})$ to the total intersection number, these 12 telltales each contribute $\big(\sigma_1(\frac{g+2}2)-1\big)$ or $\sigma_1(\frac{g+2}2)$ according to whether or not $(g+2)/2$ is a square.
  This accounts for the term $\frac12\eis_2(q^2)
  -12\cdot\frac12(\theta(q^2)-1)$.
\end{proof}

\begin{prop}\label{prop:telltale-counting-series}
  Let $L=\O(B)$ be a bisection line bundle of genus $g$. Then the number of simple telltales in the complete linear system $\abs{L}$ is the coefficient of $q^{g+2}$ in
  \[
    \eis_{4,\infty}(q)
    \quad\text{ or }\quad
    \frac12 \eis_4(q^4)+\frac12,
  \]
  according to whether $L$ is ordinary or Weierstrass, and for any $m>0$, the number of non-simple telltales in $\abs{L}$ with fiber-multiplicity $m$ is the coefficient of $q^{g+2-2m}$ in
  \[
    12\eis_{4,\infty}(q)
    \quad\text{ or }\quad
    12\left(\frac12 \eis_4(q^4)+\frac12\right),
  \]
  according to whether $L$ is ordinary or Weierstrass.
\end{prop}

\begin{proof}
  The statement about non-simple telltales follows from the one about simple telltales, since there is a 12-to-1 correspondence $s_1+s_2+mN\mapsto s_1+s_2$ between non-simple telltales with fiber-multiplicity $m$ in $\abs{L}$ and simple telltales in $\abs{L\otimes \O(-mF)}$, and by adjunction the latter line bundle has genus $g-2m$.

  For the simple telltales, we are counting the number of (unordered) pairs of sections $\set{s_1,s_2}$ in $\MW(R)$ with $s_1+s_2\in\abs{L}$.
  Using the orthogonal projection $\Pi$ from \S\ref{subsec:elliptic-surfaces}, any such pair $\set{s_1,s_2}$ gives a pair $\set{u_1,u_2}=\set{\Pi s_1,\Pi s_2}$ in $\ZF^\perp(R)\cong \E8(-1)$ with $u_1+u_2=\Pi(B)$.
  Moreover, one can check that, conversely, given $\set{u_1,u_2}$ with $u_1+u_2=\Pi(B)$, the sections $\set{s_1,s_2}=\set{\Pi\I u_1,\Pi\I u_2}$ will satisfy $s_1+s_2\in\abs{B}$ if and only if $4u_1\cdot u_2=w^2+2(g+2)$, where $w=\Pi(B)$.

  Hence, we are reduced to counting such pairs $\set{u_1,u_2}$ for fixed $w$.
  Moreover, $L$ is Weierstrass or ordinary according to whether or not $w\in 2\E8(-1)$; we say that $w$ is \emph{2-div} in the former case, and otherwise that it is \emph{2-prim}. One checks, by establishing an explicit bijection, that the number of such pairs $\set{u_1,u_2}$ is independent of the particular choice of $w$, depending only on whether $w$ is 2-prim or 2-div.

  Now the case of Weierstrass $B$ follows readily: since we can choose any 2-div $w$ we like, we might as well take $w=0$, so we are then tasked with finding the number of unordered pairs $\set{u_1,u_2}=\set{u_1,-u_1}$ with $-4u_1^2=2(g+2)$. This is equal to the $q^{(g+2)/4}$-coefficient of the theta function for $\E8$, which is equal to $\eis_4$.
  The $1/2$ factor comes from the fact that we are counting unordered pairs, and we adjust the constant term to 1 to handle the case where $B$ is a doubled section.

  The case of ordinary $B$ is more involved.
  Let $\wt\eis_{4,\infty}(q)$ be the series whose $q^{g+2}$ coefficient is the number of pairs $\set{u_1,u_2}$ with $u_1+u_2=w$ and $4u_1\cdot u_2=w^2+2(g+2)$ for any fixed 2-prim $w$ (with $w^2 \equiv 2g\mod 4$).
  We first show that $\wt\eis_{4,\infty}(q)$ is modular, whereupon it is proved to be equal to $\eis_{4,\infty}(q)$ by checking finitely many terms. It is given by the following modular form of weight 4 for $\Gamma_0(4)$:
  \begin{equation}\label{eq:e4infexpr}
  2\wt\eis_{4,\infty}(q) =
  \frac{1}{120}\left(\eis_4(q)-\eis_4(q^4)\right)_{\mathrm{odd}}+
  \frac{1}{135}\left(\eis_4(q)-\eis_4(q^4)\right)_{\mathrm{even}}
  \end{equation}
  where $f(q)_{\mathrm{even}}$ and $f(q)_{\mathrm{odd}}$ denote the power series obtained from $f(q)$ by eliminating all the odd and even terms, respectively.
  We are using the fact that $\Mod(\Gamma_0(4),k)$ is closed under $f(q)\mapsto f(q)_{\mathrm{even}}$ and $f(q)\mapsto f(q)_\mathrm{odd}$.

  The $q^{g+2}$ coefficient on the right-hand side of \eqref{eq:e4infexpr} is equal to the number of 2-prim vectors in $\E8(-1)$ of norm $-2(g+2)$, divided by 120 or 135, respectively, depending on whether $g+2$ is even or odd.
  To compute the $q^{g+2}$ coefficient of $\wt\eis_{4,\infty}$, we again use our freedom in choosing $w$, arranging $w^2=-2(g+2)$, so that we are counting pairs $\set{u_1,u_2}$ with $u_1+u_2=w$ and $u_1\cdot u_2=0$.
  By Thales' theorem, this is the same as the number of lattice points on the sphere of radius $\sqrt{-w^2}/2$ around $w/2$. This, in turn, is the number of 2-div lattice points on the sphere of radius $\sqrt{-w^2}$ around $w$, which is the same as the number of lattice points of norm $w^2=-2(g+2)$ which are equal to $w$ in $\E8/2\E8$. Note that any such lattice point is 2-prim.

  We are now tasked with counting the number of vectors of norm $-2(g+2)$ which are equal to $w$ mod $2\E8$.
  By Lemmas~\ref{lem:trans-on-roots}~and~\ref{lem:parities-classf}, this number is independent of $w$, and hence is equal to the total number of 2-prim vectors of norm $-2(g+2)$ divided by the number of equivalence classes modulo $2\E8$ of vectors of norm $-(g+2)$. Hence, it remains to see that the number of these equivalence classes is 120 or 135, according to whether $-2(g+2)$ is odd or even. But by Lemma~\ref{lem:parities-classf} again, this is the same as the number of equivalence classes modulo $2\E8$ of vectors of norm 2 or 4, which can be computed directly (with a computer).
% Sage code:
% L = RootSystem("E8")
% Q = QuadraticForm(L.cartan_matrix())
% # Get the vectors of length 2,4
% vec_lists = [list(map(vector, l)) for l in Q.vectors_by_length(2)][1:]
% # vectors_by_length (Cohn's algorithm) only gives one each pair {v,-v}
% # So get the rest.
% vec_lists = [vs + [-v for v in vs] for vs in vec_lists]
% # Get the mod-even-vectors equivalence classes of a list
% def eqclasses(vecs):
%     res = []
%     for v in vecs:
%         for vs in res:
%             if all(i%2 == 0 for i in v - vs[0]):
%                 vs.append(v)
%                 break
%         else:
%             res.append([v])
%     return res
% print("Number of classes of length 2 vectors", len(eqclasses(vec_lists[0])))
% print("Number of classes of length 4 vectors", len(eqclasses(vec_lists[1])))
\end{proof}

\begin{cor}\label{cor:decomps-exist}
  For any $u\in \E8$ and any $m\in\Z_{\ge0}$ with $u^2\equiv 2m\mod 4$, there exists a decomposition $u=v+w$ with $4v\cdot w=u^2-2m$.
\end{cor}

\begin{proof}
    This is a corollary of the proof of Proposition~\ref{prop:telltale-counting-series}, where it was shown that the number of such decompositions is the coefficient of $q^m$ in  $\eis_4$ (resp. $\eis_{4,\infty}$) when $u$ is 2-div (resp. 2-prim). Both modular forms have positive coefficients in all non-negative (resp. positive) degrees.
\end{proof}

We obtain the following generalization of Lemma~\ref{lem:g-g-bisections-biject}.
Recall that we write $\mathcal{B}_{g,n}$ for the set of arithmetic genus $g$, height $n$ bisections.

\begin{cor}\label{cor:g-n-bisections-biject}
  The restriction $\rstr{\Pi}{\mathcal{B}_{g,n}}:\mathcal{B}_{g,n}\to\ZF^\perp(R)$ is a bijection onto the set of $v\in \ZF^\perp(R)$ of norm $2g-(4n+4)$.
  In particular, $\#\mathcal{B}_{g,n}$ is the coefficient of $q^{2n+2-g}$ in $\eis_4(q)$.
\end{cor}
\begin{proof}
  As in Lemma~\ref{lem:g-g-bisections-biject}, we only need to prove surjectivity.
  Given $u\in\ZF^\perp(R)\cong \E8(-1)$ of norm $2g-(4n+4)$, there exists by Corollary~\ref{cor:decomps-exist} a decomposition $u=v+w$ with $4v\cdot w=u^2+2(g+2)$. Hence, taking $s$ and $t$ to be the unique sections with $\Pi s=v$ and $\Pi t=w$ and setting $B=s+t$ gives $\Pi(B)=u$, and we have $B\in\mathcal{B}_{g,n}$ by Propositions~\ref{prop:sec-proj-formula}~and~\ref{prop:bisec-proj-formula}.
\end{proof}

\subsection{Interpretation of the main series in terms of telltales}
We now explain the motivation for Conjecture~\ref{conj:multconj}.
The starting point is a second formula for the $q$-series found in Proposition~\ref{prop:nl-series-formula}.
\[
\varphi(q) = -\frac{1}{3} E_2 E_4^2 - \frac{2}{3} E_4E_6
\]
Observe that $\varphi(q^2)$ has the same even order terms as the series
\[
2\left(\eis_4-\eis_4(q^4)\right)\left(D\eis_{4,\infty} - \eis_{4,\infty} \eis_2(q^2)\right) + \eis_4(q^4)\left(D\eis_4(q^4)-\eis_4(q^4)\eis_2(q^2)\right).
\]
This is a finite check because both sides are quasi-modular forms of weight 10 for $\Gamma_0(4)$.

We recognize parts of this expression from Proposition~\ref{prop:conjectural-degree-formula}.
To make this more explicit, let us rewrite it as:
\begin{equation}\label{eq:full-phiq2-formula}
\begin{split}
&2\big(
\eis_4-\eis_4(q^4)\big)
\big(
D\eis_{4,\infty} - \eis_{4,\infty} \eis_2(q^2)
\big) \\
+& \eis_4(q^4)\Big(
    D\eis_4(q^4)-\eis_4(q^4)\eis_2(q^2)-12\big(\theta(q^2)-1\big)
\Big)\\
+&(1-E_4(q^4))
+12E_4(q^4)\big(\theta(q^2)-1\big)-1.
\end{split}
\end{equation}
By Corollary~\ref{cor:g-n-bisections-biject}, the number of ordinary or Weierstrass bisection line bundles of height $n$ and genus $g$ is given by the $q^{(2n+4)-(g+2)}$ coefficient of the series $E_4-E_4(q^4)$ or $E_4(q^4)$, respectively.
Hence, using Proposition~\ref{prop:conjectural-degree-formula}, the first two lines of \eqref{eq:full-phiq2-formula} correspond to the degree-counting series in Proposition~\ref{prop:degrees-are-modular}, and the third line of \eqref{eq:full-phiq2-formula} are precisely the correction terms \eqref{eq:excess-contr} and \eqref{eq:nodal-contr}.

The multiplicity rule in Conjecture~\ref{conj:multconj} is the simplest expression that depends only upon $m$, and which recovers $\varphi(q)$ via \eqref{eq:full-phiq2-formula}.

\bibliographystyle{alpha}
\bibliography{\jobname}
\end{document}